\documentclass{elsarticle}
\usepackage{a4wide}

\usepackage{amsmath}
\usepackage{dsfont}

\usepackage{color}
\usepackage{amsmath,amssymb,bbm,enumerate}
\usepackage{graphicx}
\usepackage[Symbol]{upgreek}

\newtheorem{proposition}{Proposition}
\newtheorem{lemma}{Lemma}
\newtheorem{definition}{Definition}

\newtheorem{theorem}{Theorem}
\newtheorem{remark}{Remark}
\newtheorem{example}{Example}

\newenvironment{proof}{\medskip\noindent{\bf Proof.}\;}{\null\hfill $\Box$\par\medskip }

\usepackage{commath}
\usepackage{bm}
\usepackage{isomath}
\usepackage{xfrac}

\renewcommand{\geq}{\geqslant}
\renewcommand{\leq}{\leqslant}

\let\oldpagenumbering\pagenumbering
\renewcommand{\pagenumbering}[1]{%
	\cleardoublepage
	\oldpagenumbering{#1}
}

\newcommand{\diff}{\mathop{}\mathopen{}\mathrm{d}}
\newcommand{\norme}[1]{\left\lVert#1\right\rVert}

\newcommand{\eu}{\ensuremath{\mathrm{e}}}

\newcommand{\jj}{\ensuremath{\mathrm{j}}}

\newcommand{\T}{\ensuremath{\mathsf{T}}}

\newcommand{\expj}[1]{\ensuremath{\eu^{\, \jj #1}}}

\providecommand{\renewoperator}[3]{%
	\renewcommand*{#1}{\mathop{#2}#3}}
\renewoperator{\Re}%
	{\mathrm{Re}}{\nolimits}
\renewoperator{\Im}%
	{\mathrm{Im}}{\nolimits}
	
\newcommand{\vect}[1]{\bm{#1}}
\newcommand{\mat}[1]{\mathbf{#1}}

\newcommand{\ope}[1]{\mathbf{#1}}
\newcommand{\ip}[2]{\left \langle #1,\, #2\right \rangle}
\newcommand{\RT}{\mathcal{R}}
\newcommand{\RTb}{\bm{\RT}}

\def\R{\mbox{$\mathbb{R}$}}

\def\S{\mbox{$\mathbb{S}$}}

\def\C{\mbox{$\mathbb{C}$}}

\def\vx{\vect{x}}

\def\vh{\vect{h}}

\def\vxi{\vect{\xi}}

\usepackage{hyperref}
\begin{document}

\title{Riesz-based orientation of localizable Gaussian fields\tnoteref{t1}}
%\tnotetext[t1]{The authors acknowledge the support of the French Agence Nationale de la Recherche (ANR) under reference ANR- 13-BS03-0002-01 (ASTRES). The last author has been supported by the ERC project AdG-2013-320594 (DECODA).}

\author[ljk]{K. Polisano\corref{cor1}}
\ead{Kevin.Polisano@univ-grenoble-alpes.fr}

\author[ljk]{M. Clausel}
\ead{Marianne.Clausel@univ-grenoble-alpes.fr}

\author[ljk]{V. Perrier}
\ead{Valerie.Perrier@univ-grenoble-alpes.fr}

\author[gipsa]{L. Condat}
\ead{Laurent.Condat@gipsa-lab.grenoble-inp.fr}

\address[ljk]{Univ. Grenoble Alpes, CNRS, Grenoble INP, LJK, 38000 Grenoble, France}
\address[gipsa]{Univ. Grenoble Alpes, CNRS, GIPSA-lab, 38000 Grenoble France}

\cortext[cor1]{Corresponding author}

\begin{abstract}
In this work we give a sense to the notion of orientation for self-similar Gaussian fields with stationary increments, based on a Riesz analysis of these fields, with isotropic zero-mean analysis functions. We propose a structure tensor formulation and provide an intrinsic definition of the orientation vector as eigenvector of this tensor. That is, we show that the orientation vector does not depend on the analysis function, but only on the anisotropy encoded in the spectral density of the field. Then, we generalize this definition to a larger class of random fields called localizable Gaussian fields, whose orientation is derived from the orientation of their tangent fields. Two classes of Gaussian models with prescribed orientation are studied in the light of these new analysis tools.
\end{abstract}


%\date{}

\maketitle
%%%%%%%%%%%%%%%%%%%%%%%

\section{Introduction}
Anisotropic images, admitting different characteristics along a considered direction, are ubiquitous in many areas as computer vision~\cite{Pey07}, image processing~\cite{BE01}, and hydrology~\cite{Ben06}. A major issue is then the definition of a suitable concept of local anisotropy. 

A widely used approach, in the image processing community, consists in defining directionality properties of an image by means of its Riesz transform~\cite{FS01}. Several characteristics can then be derived from the knowledge of the Riesz transform of an image: its local orientation, which is roughly speaking the dominant direction at a given point and the structure tensor \cite{bigun1987optimal} whose rank is related to the local dimensionality of the image. This approach has proved to be successful for many applications such as classification or texture retrieval \cite{DFVM11}. Recently, this framework has been extended to the case of superimposed patterns. An extension of the synchrosqueezing method to the bidimensionnal setting, based on wavelet analysis, has been proposed in~\cite{COP15}. 

In many cases, the analyzed anisotropic image is related to some physical phenomena, that can be well modeled using a stochastic approach. Anisotropic random fields are then naturally involved in the modeling of medical images~\cite{BE01} or in spatial statistics~\cite{DH99}. In such situations, the Riesz framework is not so easy to apply. The main difficulty lies in giving a rigourous definition of the Riesz transform of a random field. Indeed~\cite{Stein70}, the Riesz transform of a function is well defined if it belongs to $L^p$ for some $p>1$, which is not the case for the sample paths of many classical random fields, like Fractional Fields widely used to model random textures. The nonlocal character of the Riesz transform then prevents any definition based on a restriction of the considered random field to a compact set. 

To overcome all these difficulties, we choose to use a Riesz-based approach, adapting the work of~\cite{bigun1987optimal,F02} about Riesz-analysis of anisotropic images characteristics. In \cite{F02}, the authors defined a structure tensor and an orientation of the neighborhood of an image, derived from the Riesz framework. Numerical experiments have put in evidence the effectiveness of the approach and especially the fact that one clearly recovers the anisotropic features of the image. From the theoretical point of view, the orientation and the structure tensor depend on the chosen analysis function. We show that, surprisingly, considering the very general case of localizable Gaussian fields, the anisotropic characteristics of a random field become intrinsic: neither the structure tensor nor the orientation vector depend on the analysis function, but only on the anisotropy encoded in the density function of its tangent fields. 

Our paper is organized as follows. In Section~\ref{s:riesz}, we first recall some basic facts about the Riesz transform and its use for defining anisotropic features of an image. Then in Section~\ref{s:self-similar}, we define the Riesz-based notion of the orientation and the structure tensor in the Gaussian self-similar with stationary increments case. We prove in Theorem~\ref{pro:wav:based:orientation} that these two characteristics are intrinsic in the sense that they depend only on the anisotropic properties of the analyzed random field. Section~\ref{s:wav:based:orientation:loc} is then devoted to the extension of all these notions to the localizable case. We then provide two classes of Gaussian models with prescribed orientation. For sake of clarity, we postponed all proofs are postponed in Section~\ref{s:proofs}.
%%%%%%%%%%%%%%%%%%%%%%%%%%%%%%%%%%%%%%%%%%%%%%%%%%%%%%%%%%%

\section{Classical tools in directionality analysis of images}\label{s:riesz}
In this section, we give some background about two classical tools for analyzing  the anisotropy properties of an oriented texture: the local orientation and the structure tensor. We first recall in Section~\ref{s:local-orientation} the usual definitions based on the Riesz transform introduced in \cite{FS01}. %Thereafter, Section~\ref{s:IMRA} is devoted to the presentation of the wavelet extension of these two notions based on monogenic wavelet analysis (see \cite{UC13}). 

%%%%%%%%%%%%%%%
\subsection{Local orientation of an image and structure tensor\\}\label{s:local-orientation}

The classical notion of local orientation of a texture is based on the Riesz transform. The Riesz transform $\RTb f$ of any $f\in L^2(\mathbb{R}^2)$ is defined in the Fourier domain\footnote{where the 2D Fourier transform is defined for $f\in L^1 (\mathbb{R}^2)$ by $\widehat f(\vxi)=\int_{\R^2} f(\vx) \eu^{-\jj \ip{\vx}{\vxi}} \dif \vx$ and then extended by the Plancherel theorem and by continuity arguments for $f\in L^2 (\mathbb{R}^2)$.} as
\[
\RTb f=\begin{pmatrix}\RT_1f\\\RT_2f\end{pmatrix}\quad \text{with}\quad \widehat{\RT_1f}(\vect{\xi})=-\jj \frac{\xi_1}{\norme{\vect{\xi}}}\widehat{f}(\vect{\xi}),\quad \widehat{\mathcal{R}_2f}(\vect{\xi})=-\jj \frac{\xi_2}{\norme{\xi}}\widehat{f}(\vect{\xi})~,~~\forall  \vect{\xi} =(\xi_1, \xi_2) \in \R^2\;.
\]
The main properties of $\RTb$ \cite{Stein70,UV09} are summarized in the two following propositions. The first ones concern the invariance with respect to dilations, translations, and the steerability property (relation with the rotations).
\begin{proposition}\label{pro:dil}
The Riesz transform commutes both with the translation, and the dilation operator, that is for any $f\in L^2(\mathbb{R}^2)$, $a>0$ and $\vect{b}\in \mathbb{R}^2$, one has
\[
\RTb \ope{D}_a f=\ope{D}_a\RTb f\quad \text{with}\quad \ope{D}_a f=f(a^{-1} \cdot)\;,
\]
and\
\[
\RTb \ope{T}_{\vect{b}} f=\ope{T}_{\vect{b}}\RTb f\quad \text{with}\quad \ope{T}_{\vect{b}} f=f(\cdot-\vect{b})\;.
\]
\end{proposition}
\begin{proposition}\label{pro:steer}
The Riesz transform is steerable, that is, for any $f\in L^2(\mathbb{R}^2)$ one has
\[
{R}_\theta(\RTb f)=\ope{R}_\theta^{-1}\RTb ({R}_\theta f)=\begin{pmatrix}\cos\theta\, \RT_1({R}_\theta f)+\sin\theta\, \RT_2({R}_\theta f)\\-\sin\theta\, \RT_1({R}_\theta f)+\cos\theta\, \RT_2({R}_\theta f)\end{pmatrix}\;,
\] 
where ${R}_\theta f=f(\ope{R}_{-\theta}\, \cdot)$ is the rotation operator by the angle $\theta$, and 
\[
\ope{R}_{-\theta}=\begin{pmatrix}\cos\theta&\sin\theta\\ -\sin\theta&\cos\theta\end{pmatrix}\;.
\]
is the matrix of the spatial rotation of angle $-\theta$.
\end{proposition}
The Riesz transform is also a unitary and componentwise antisymmetric operator on $L^2(\mathbb{R}^2)$.
\begin{proposition}\label{pro:unit}
For any $i\in\{1,2\}$, the $i$--th component of the Riesz transform $\RT_i$ is an antisymmetric operator, namely for all $f,g\in L^2(\mathbb{R}^2)$ we have
\[
\ip{\RT_i f}{g}_{L^2(\mathbb{R}^2)}=-\ip{f}{\RT_i g}_{L^2(\mathbb{R}^2)}\;.
\]
and ${\RT_1}^2+{\RT_2}^2=-\ope{I}$, which imply in particular that
\[
\ip{\RTb f}{\RTb g}_{L^2(\mathbb{R}^2,\mathbb{R}^2)}=\ip{\RT_1 f}{\RT_1 g}_{L^2(\mathbb{R}^2)}+\ip{\RT_2 f}{\RT_2 g}_{L^2(\mathbb{R}^2)}=
\ip{f}{g}_{L^2(\mathbb{R}^2)}\;.
\]
\end{proposition}
Using the Riesz transform one can also define the so-called orientation of an image \cite{FS01,YQS11}.
\begin{definition}[Local orientation of a deterministic function]
Let $f\in L^2(\mathbb{R}^2, \mathbb{R}^*_+)$. \\ Assume that a.e. $\RTb f\neq 0$. One then defines 
\[
\forall a.e. ~\vect{x}\in\mathbb{R}^2,\quad \vect{\vec n}(\vx)=\frac{\RTb f(\vx)}{\norme{\RTb f(\vx)}}\;,
\]
the local orientation of $f$ at point $\vect{x}$.
\end{definition}
The Riesz transform can be viewed as a smooth version of the gradient operator. To characterize the degree of directionality of $f$ at some point $\vect{x}$, a classical tool, widely used in the image processing community, is the so-called structure tensor involving the gradient \cite{bigun1987optimal,harris1988combined,jahne2004practical}, which has been revisited through the Riesz transform by \cite{F02}. Let us begin with the $2\times 2$ matrix:
\begin{equation}\label{eq:Jf}
\ope{J}_f(\vect{x})=\RTb f(\vect{x})\RTb f^{\T}(\vect{x})=\begin{pmatrix}\mathcal{R}_1 f(\vect{x})^2&\mathcal{R}_1 f(\vect{x})\mathcal{R}_2 f(\vect{x})\\ \mathcal{R}_2 f(\vect{x})\mathcal{R}_1 f(\vect{x})&\mathcal{R}_2 f(\vect{x})^2\end{pmatrix}\;.
\end{equation}
This matrix is symmetric, positive definite, of rank one and then admits 0 as eigenvalue and also 1 with the associated eigenvector $\RTb f(\vect{x})$, that is the local orientation, since 
\[
\ope{J}_f(\vect{x})\RTb f(\vect{x})=(\RTb f(\vect{x})\RTb f(\vect{x})^{\T})\RTb f(\vect{x})=\RTb f(\vect{x})(\RTb f(\vect{x})^{\T}\RTb f(\vect{x}))=\RTb f(\vect{x})\;.
\] 
In practice, we define the local orientation vector on a point $\vect{x}\in \mathbb{R}^2$ as the vector minimizing the distance to the set of the Riesz vectors $\RTb f (\vect{x}')$ on a neighborhood of $\vect{x}\in \mathbb{R}^2$, defined by a positive window $W(\vect{x}-\vect{x}')$, that is we are looking for the vector $\vect{n}$ maximizing the quantity $ \int W(\vect{x}-\vect{x}')(\RTb f(\vect{x}')^\T\vect{\vec n})^2\diff \vect{x}'$ or equivalently $\vect{\vec n}^\T\mat{J}_f^W(\vect{x})\vect{\vec n}$, with
\begin{equation}\label{eq:JfW}
\ope{J}_f^W(\vect{x})= (W*\ope{J}_f)(\vect{x})\;,
\end{equation}
which is the matrix $\ope{J}_f(\vect{x})$ filtered by the positive windows function $W$ to form the structure tensor \cite{bigun1987optimal,harris1988combined,jahne2004practical}. This matrix is symmetric and positive definite, then admits two nonnegative eigenvalues $\lambda_1(\vx)$ and $\lambda_2(\vx)$. It is easy to see that the local orientation $\vect{\vec n}(\vect{x})$ maximizing $\vect{\vec n}^\T\mat{J}_f^W(\vect{x})\vect{\vec n}$ is always an eigenvector of the matrix $\ope{J}_f^W(\vect{x})$ associated to its largest eigenvalue $\lambda_1(\vx)$. 

The following coherency index provides a degree of directionality at any point \cite{bigun1987optimal}:
\[
\chi_f(\vect{x})=\frac{{\lambda_{1}(\vect{x})- \lambda_{2}(\vect{x})}}{\lambda_{1}(\vect{x})+\lambda_{2}(\vect{x})}\in [0,1]\;.
\]
The case $\chi_f(\vect{x})\approx1$ corresponds to an almost one dimensional image at $\vect{x}$, whereas the case $\chi_f(\vect{x})=0$ may correspond to different situations, such as isotropy or existence of a corner. Note that, all these quantities depend on the chosen windows $W$.

The aim of next sections is to adapt this framework to the case of random Gaussian fields. We first shall consider the case of self-similar Gaussian fields admitting stationary increments in Section \ref{s:self-similar}, and in Section~\ref{s:wav:based:orientation:loc} to the more general classe of localizable Gaussian fields. 

%%%%%%%%%%%%%%%%%%%%%%%%
\section{Riesz-based orientation of self-similar Gaussian fields with stationary increments}\label{s:self-similar}

We observe first that all the definitions of the previous section cannot be extended directly to the case of random fields. Indeed, there is a difficulty to circumvent, which is the definition of the Riesz transform of a Gaussian random field, like a Brownian motion for example. The usual definition of the Riesz transform holds for $L^p$ functions. Unfortunately, the sample paths of many classical random Gaussian fields do not belong to these spaces. Defining the Riesz transform of a Gaussian field by duality (as it is often done for the transforms of distributions) is not straightforward, because the Schwartz class is not stable by the Riesz transform. Although some work on the Hilbert transform of temperate distributions \cite{ishikawa1985generalized,koizumi1959hilbert} can potentially be extended to a generalized Riesz transform adapted to random fields, we favor here to define a notion of orientation for Gaussian fields, which is easily interpretable (intuitively related to the spectral density of the field) and easy to manipulate from a computational point of view (which can be explicitly computed for usual random fields and estimated by the use of wavelets or filter bank). 

In this section, we first begin with the simple case of self-similar Gaussian fields. We give in Section~\ref{s:self-similar:def} some background on such fields. Then in Section~\ref{s:wav:based:orientation:ssi}, we define our notion of orientation, based on a Riesz analysis of these fields with an appropriated analysis function. Finally in Section~\ref{s:examples}, we give several examples of oriented self-similar Gaussian fields with stationary increments. From now on, we restrict ourselves to bidimensionnal centered real valued Gaussian fields, since our goal is to analyse anisotropic images. We shall also assume that the Gaussian field $X$ under consideration is stochastically continuous; that is, the covariance
\[
(\vect{x},\vect{y})\mapsto \mathbb{E}[X(\vect{x})X(\vect{y})]\;,
\]
is a continuous function on $(\mathbb{R}^2)^2$.
%%%%%%%%%%%%%%%%%%%%%
\subsection{Self-similar Gaussian fields with stationary increments\\}\label{s:self-similar:def}

In what follows we shall focus on the special case of $H$--self-similar Gaussian fields admitting stationary increments ($H$-sssi), studied in \cite{Do79, BJR97, bonami2003anisotropic}. Remember that the bidimensionnal Gaussian field $X$ is said to {\it admit stationary increments} if for any $\vx\in\mathbb{R}^2$,
\[
\{X(\vx+\vh)-X(\vx)\}_{\vh\in\mathbb{R}^2}\overset{(fdd)}{=}\{X(\vh)\}_{\vh\in\mathbb{R}^2}\;,
\]
whereas $X$ is said to be {\it $H$--self-similar} (see~\cite{Do79}), for some $H\in (0,1)$ if
\[
\forall c>0,\quad \{X(c\vx)\}_{\vect{x}\in\mathbb{R}^2}\overset{(fdd)}{=}\{ c^H X(\vx)\}_{\vx\in\mathbb{R}^2}\;,
\]
where as usual $\overset{(fdd)}{=}$ means equality of finite dimensional distributions. Since $X$ is assumed to be stochastically continuous, the self-similarity implies in particular that $X(\vect{0})=0$ a.s.

We now recall, following~\cite{yaglom1957some}, the notion of spectral measure of a Gaussian field admitting stationary increments, based on the following classical result.
%%%%%%%%%
\begin{proposition}\label{prop:borel}
Let $X=\{X(\vect{x})\}_{\vect{x}\in \mathbb{R}^2}$ be a centered real-valued Gaussian field with stationary increments. Then, there exists a unique Borel measure $\sigma_X$ satisfying
\[
\int_{\mathbb{R}^2} \min(1, \|\vect{\xi}\|^2)\dif \sigma_X(\vect{\xi})<\infty\;,
\] 
%(1\wedge |\vect{\xi}|)
such that for any $\vect{x},\vect{y}\in \mathbb{R}^2$, the covariance reads:
\[
\mathbb{E}(X(\vect{x})X(\vect{y}))=\int_{\mathbb{R}^2}(\expj{\ip{\vect{x}}{\vect{\xi}}}-1)(\eu^{-\jj \ip{\vect{y}}{\vect{\xi}}}-1)\dif \sigma_X(\vect{\xi})\;.
\]
The measure $\sigma_X$ is called the spectral measure of the Gaussian field with stationary increments $X$.
\end{proposition}
In what follows we shall consider only Gaussian fields whose spectral measure $\sigma_X$ admits a density $f_X$, called the spectral density of $X$, with respect to the Lebesgue measure: $\dif \sigma_X(\vect{\xi})=f_X(\vect{\xi}) \dif \vect{\xi} $. Since $X$ is real-valued this function is necessarily even. Such field admits an harmonizable representation:
\begin{equation}
\label{eq:harmonizable}
X(\vx)=\int_{\mathbb{R}^2}\left(\expj{\ip{\vect{x}}{\vect{\xi}}}-1\right)f_X^{1/2}(\vect{\xi})\widehat{\ope{W}}(\diff \vect{\xi})\;,
\end{equation}
where $\widehat{\ope{W}}$ is a complex-valued white noise.
By uniqueness of the spectral density, the representation of $H$--self-similar Gaussian fields follows (see also~\cite{Do79, BJR97, bonami2003anisotropic, cohen2013fractional}):
\begin{proposition}[Dobrushin \cite{Do79}]\label{pro:spectral-density-ss}
Let $H\in (0,1)$ and $X$ be a $H$--self-similar Gaussian field with stationary increments admitting a spectral density $f_X$. Then $f_X$ is of the form
\[
f_X(\vect{\xi})=\norme{\vect{\xi}}^{-2H-2} C_X\left(\frac{\vect{\xi}}{\norme{\vect{\xi}}}\right)\;,
\]
where $C_X$ is a positive homogeneous function defined on the sphere $\S^{1}=\{\vect{\xi}\in\mathbb{R}^2,\,\norme{\vect{\xi}}=1\}$. The function $C_X$ is called the anisotropy function of $X$.
\end{proposition}
\begin{remark}
The estimation problem of the anisotropy function has been addressed by Istas in~\cite{Ist07}.
\end{remark} 
%\begin{proof}
%By definition, for any $c>0$,
%\[
%\{X(c\vx)\}_{\vx\in\mathbb{R}^2}\overset{(fdd)}{=}\{ c^H X(\vx)\}_{\vx\in\mathbb{R}^2}\;.
%\]
%It implies in particular that
%\[
%\forall \vect{x}, \vect{y}\in \mathbb{R}^2, \quad \mathbb{E}(X(c\vect{x})X(c\vect{y}))=c^{2H}\mathbb{E}(X(\vect{x})X(\vect{y}))\;.
%\]
%Since $X$ is a Gaussian field with spectral density $f_X$, following Proposition \ref{prop:borel}, we obtain:
%\[
%\int_{\mathbb{R}^2} \abs{\eu^{\, \jj c\ip{\vect{x}}{\vect{\xi}}}-1}^2f_X(\vect{\xi})\dif \vect{\xi} =c^{2H}\int_{\mathbb{R}^2} \abs{\eu^{\, \jj \ip{\vect{x}}{\vect{\xi}}}-1}^2f_X(\vect{\xi})\dif \vect{\xi}\;.
%\]
%Performing the change of variable $\vect{\xi}\leftarrow c\vect{\xi}$ in the first integral, and by uniqueness of the spectral density, one then deduces:
%\[
%\mbox{for a.e. } \vect{\xi}\in \mathbb{R}^2,\quad f_X\left(\frac{\vect{\xi}}{c}\right)=c^{2H+2}f_X(\vect{\xi})\;.
%\]
%Fix now $\vect{\xi}\neq 0$ and set $c=\norme{\vect{\xi}}$. One deduces directly the required result, namely that $f_X$ has to be known only on $\mathbb{S}^1$ to be defined almost everywhere, and that $C_X$ is defined as $\vect{\xi}\mapsto\norme{\vect{\xi}}^{2H+2}f_X(\vect{\xi})$.
%\end{proof}

\noindent 
We now investigate the orientation properties  of a self-similar Gaussian field deformed by a linear transform.
\begin{proposition}\label{pro:linear}
Let $X$ be a $H$--self-similar Gaussian field with stationary increments admitting as spectral density $f_X$ and as anisotropy function $C_X$. Let $\mat{L}$ be an invertible $2\times 2$ real valued matrix. Define $X_{\mat{L}}$ by $X_{\mat{L}}(\vect{x})=X(\mat{L}^{-1} \vect{x})$. Then $X_\mat{L}$ is a $H$--self-similar Gaussian field admitting as
\begin{itemize}
\item spectral density 
\[
f_{X_\mat{L}}(\vect{\xi})=\abs{\det(\mat{L})} f_X(\mat{L}^{\T}\vect{\xi})\;, \quad \vect{\xi} \in \R^2\;,
\]
\item anisotropy function 
\[
C_{X_\mat{L}}(\vect{\Theta})=\frac{\abs{\det(\mat{L})}}{\norme{\mat{L}^{\T} \vect{\Theta}}^{2H+2}}C_X\left(\frac{\mat{L}^{\T}\,\vect{\Theta}}{\norme{\mat{L}^{\T}\,\vect{\Theta}}}\right)\;, \quad \vect{\Theta}\in\mathbb{S}^1.
\]
\end{itemize}
\end{proposition}
\begin{proof}
The self-similarity and stationarity properties of $X_\mat{L}$ directly come from that of $X$ and of the linearity of $\mat{L}$. To compute the spectral density of $X_\mat{L}$, observe that:
\begin{align*}
\mathbb{E}(X_\mat{L}(\vect{x}),X_\mat{L}(\vect{y}))
&= \mathrm{Cov}(X(\mat{L}^{-1}\vect{x}),X(\mat{L}^{-1}\vect{y}))\;,\\
&=\int_{\mathbb{R}^2}\left(\expj{\ip{\mat{L}^{-1}\vect{x}}{\vect{\xi}}}-1\right)\left(\eu^{-\jj \ip{\mat{L}^{-1}\vect{y}}{\vect{\xi}}}-1\right)f_X(\vect{\xi})\dif \vect{\xi}\;,\\
&=\int_{\mathbb{R}^2}\left(\expj{\ip{\vect{x}}{(\mat{L}^{-1})^\T\vect{\xi}}}-1\right)\left(\eu^{-\jj \ip{\vect{y}}{(\mat{L}^{-1})^\T \vect{\xi}}}-1\right)f_X(\vect{\xi})\dif \vect{\xi}\;,\\
&=\int_{\mathbb{R}^2}\left(\expj{\ip{\vect{x}}{\vect{\zeta}}}-1\right)\left(\eu^{-\jj \ip{\vect{y}}{\vect{\zeta}}}-1\right)f_X(\mat{L}^\T \vect{\zeta})~
\lvert \mathrm{det}\, \mat{L}\rvert \dif \vect{\zeta}\;,
\end{align*}
using the change of variable $\vect{\vect{\zeta}}= (\mat{L}^{-1})^{\T} \vect{\xi}$, which directly leads by identification to the explicit expression of $f_{X_\mat{L}}$, as well as that of $C_{X_\mat{L}}$ given in Proposition~\ref{pro:linear}. 
\end{proof}

\noindent We now explain how to define in a proper way the directional characteristics of a Gaussian field $X$ admitting stationary increments in the self-similar case, and the relation to its anisotropy function $C_X$.   

%%%%%%%
\subsection{Definition of a Riesz-based orientation in the self-similar case\\}\label{s:wav:based:orientation:ssi}

%Let $\{\psi_{i,\vect{k}}\}$ be an isotropic tight wavelet frame as defined in Section~\ref{s:IMRA}, and $\{\psi_{i,\vect{k}}^{(\mathcal{R})}\}$ the corresponding vector valued Riesz-based wavelet tight frame generated by $\mathcal{R}\psi$.
%Notice that our theoretical framework hold for any isotropic multiresolution analysis. In particular, we can consider the framework defined in Section~\ref{s:IMRA}, but we emphazis that our results hold in whole generality. 
Our notion of Riesz-based orientation of a self-similar Gaussian field will be based on the following preliminary result, leading to a new formulation for the structure tensor.
\begin{theorem}\label{pro:wav:based:orientation}
Let $X$ be a $H$--self-similar Gaussian field admitting a spectral density $f_X$ whose anisotropy function $C_X$ is bounded and $\psi$ a real isotropic wavelet with 2 zero moments and fast decay, defined by its 2-D Fourier transform $\widehat \psi(\vect{\xi})=\varphi(\norme{\vect{\xi}})$. %and such as
%\[
%\int_{ \mathbb{R}^2}^{}|\widehat{\mathcal{R}_{\ell}\psi}(\vect{\xi})|^2 f_X(\vect{\xi})\diff \vect{\xi}<+\infty, \quad \ell\in \{1,2\}\;.
%\]
By analogy with the structure tensor \eqref{eq:Jf}-\eqref{eq:JfW} evaluated at the origin, let us define the matrix:
\[
\mat{J}_X^{\psi}= \begin{pmatrix} |\ip{X}{\mathcal{R}_1 \psi}|^2 & \ip{X}{\mathcal{R}_1 \psi}\overline{\ip{X}{\mathcal{R}_2 \psi}} \\  \overline{\ip{X}{\mathcal{R}_1 \psi}}\ip{X}{\mathcal{R}_2 \psi}  &  |\ip{X}{\mathcal{R}_2 \psi}|^2\end{pmatrix}\;,
\]
with the Gaussian variables $ \ip{X}{\mathcal{R}_{\ell} \psi} = \int_{ \mathbb{R}^2}^{} X(\vect{x}) \mathcal{R}_{\ell} \psi(\vect{x})\diff \vect{x}$. \\
Then, the covariance matrix of the Gaussian vector $(\ip{X}{\mathcal{R}_1 \psi},\ip{X}{\mathcal{R}_2 \psi})^{\T}$ is
\[
\mathbb{E} \left[\mat{J}_{X}^{\psi} \right]=\left(\int_{0}^{+\infty}\frac{| \varphi(r)|^2}{r^{2H+1}}\dif r\right)\mat{J}_X\;,
\]
where  for any $\ell_1,\ell_2\in (\{1,2\})^2$,
\begin{equation}
\label{def:newtensor}
[\mat{J}_X]_{\ell_1,\ell_2}=\int_{\vect{\Theta}\in\mathbb{S}^{1}} {\Theta}_{\ell_1}{\Theta}_{\ell_2}~C_X(\vect{\Theta})\dif \vect{\Theta}\;,
\end{equation}
with the notation $\vect{\Theta}=(\Theta_1,\Theta_2)$.\\
$\mat{J}_X$ is a nonnegative definite $2\times2$ matrix depending only on the anisotropy function $C_X$, and will be called the 
\underline{structure tensor} of $X$. 
\end{theorem}
\begin{proof} The assumptions on $\psi$ insure that the function $g_{\ell}$ defined by $g_{\ell}(\vect{\xi})=\overline{\widehat{\mathcal{R}_{\ell}\psi}(\vect{\xi})}f_X^{1/2}(\vect{\xi})$ is square-integrable; that is, the stochastic integral $\int_{ \mathbb{R}^2}^{} g_{\ell}(\vect{\xi})\ope{\widehat{W}}(\diff \vect{\xi})$ is well defined and the following operations lead to the a.s. existence of the Gaussian variables $\ip{X}{\mathcal{R}_{\ell}\psi}$, for $\ell =1,2$:
\begin{align*}
\int_{\mathbb{R}^2}
\overline{\widehat{\mathcal{R}_{\ell}\psi}(\vect{\xi}) }
f_X^{1/2}(\vect{\xi})\ope{\widehat{W}}(\diff \vect{\xi})
&=\int_{\mathbb{R}^2}
\overline{ \left(\widehat{\mathcal{R}_{\ell}\psi}(\vect{\xi})-\widehat{\mathcal{R}_{\ell}\psi}(\vect{0})\right)}
 f_X^{1/2}(\vect{\xi})\widehat{\ope{W}}(\diff \vect{\xi})\;,\\
 &=\int_{\mathbb{R}^2} \left(\int_{\mathbb{R}^2}\expj{\ip{\vect{x}}{\vect{\xi}}}
\overline{\mathcal{R}_{\ell}\psi(\vect{x})}
\diff \vect{x}-\int_{\mathbb{R}^2}
\overline{\mathcal{R}_{\ell}\psi(\vect{x})}
\diff \vect{x}\right)f_X^{1/2}(\vect{\xi})\widehat{\ope{W}}(\diff \vect{\xi})\;,\\
&=\int_{\mathbb{R}^2}\left(\int_{\mathbb{R}^2}\left(\expj{\ip{\vect{x}}{\vect{\xi}}}-1\right)f_X^{1/2}(\vect{\xi})\widehat{\ope{W}}(\diff \vect{\xi})\right)
\overline{\mathcal{R}_{\ell}\psi(\vect{x})}
\diff \vect{x}\;,\\
&=\ip{X}{\mathcal{R}_{\ell}\psi}\;.
\end{align*}
The first equality follows since $\widehat{\mathcal{R}_{\ell}\psi}(\vect{0})=0$ by the moment assumption on $\psi$ and the third equality comes from the classical stochastic Fubini Theorem (see \cite{pipiras2010regularization}), which holds since
\begin{align*}
&\int_{\mathbb{R}^2} \left(\int_{\mathbb{R}^2}|\expj{\ip{\vect{x}}{\vect{\xi}}}-1|^2 f_X(\vect{\xi})|\psi(\vect{x})|^2\diff\vect{\xi}\right)^{1/2}\diff \vect{x}\\
&=\int_{\mathbb{R}^2} \left(\int_{\mathbb{R}^2}|\expj{\ip{\vect{x}}{\vect{\xi}}}-1|^2 f_X(\vect{\xi})\diff\vect{\xi}\right)^{1/2}|\psi(\vect{x})|\diff \vect{x}<\infty\;,
\end{align*}
by the integrability properties of $\psi$ and the existence of the stochastic integral defining $X$. 

%Hence
%\[
%c_{i,\vect{k}}^{(\ell)}(X)=\int_{\mathbb{R}^2}\widehat{\mathcal{R}_{\ell}\psi_{i,\vect{k}}}(\vect{\xi}) f^{1/2}(\vect{\xi})\widehat{\ope{W}}(\diff \vect{\xi})\;.
%\]
The covariance matrix is then easily computed:
by definition of the Riesz transform, one has for any  $\ell_1,\ell_2 \in(\{1,2\})^2$,
\[
\mathbb{E}[\ip{X}{\mathcal{R}_{\ell_1}\psi}
\overline{\ip{X}{\mathcal{R}_{\ell_2}\psi}}]=\int_{\mathbb{R}^2} \frac{{\xi}_{\ell_1}{\xi}_{\ell_2}}{\norme{\vect{\xi}}^2}~\abs{\widehat{\psi}(\vect{\xi})}^2 f_X(\vect{\xi})\dif \vect{\xi}\;.
\]
We set $\vect{\xi}=r\vect{\Theta}$ with $(r,\vect{\Theta})\in \mathbb{R}^+\times \mathbb{S}^{1}$. Hence, remembering that $\psi$ is isotropic, that is $\widehat \psi (\vect{\xi})= \varphi(r)$ and using	 the expression of $f_X$ given in Proposition \ref{pro:spectral-density-ss}, we obtain
\begin{align*}
\mathbb{E}[\ip{X}{\mathcal{R}_{\ell_1}\psi}
\overline{\ip{X}{\mathcal{R}_{\ell_2}\psi}}]
&=\int_{\vect{\Theta}\in\mathbb{S}^{1}}\int_{r=0}^{+\infty}\frac{1}{r^{2H+2}}~{\Theta}_{\ell_1}{\Theta}_{\ell_2}~\abs{ \varphi(r)}^2C_X(\vect{\Theta}) ~r\dif r  \dif\vect{\Theta}\;,\\
&=\left[\int_{r=0}^{+\infty}\frac{\lvert  \varphi(r)\rvert ^2}{r^{2H+1}}\dif r\right]
\left[\int_{\vect{\Theta}\in\mathbb{S}^{1}} {\Theta}_{\ell_1}{\Theta}_{\ell_2}~C_X(\vect{\Theta})\dif \vect{\Theta}\right]\;.
\end{align*}
Theorem~\ref{pro:wav:based:orientation} then follows.
\end{proof}

%%%%%%%%%%%%%
\noindent We now define the structure tensor of $X$, its orientation and its coherency index.
\begin{definition}[Orientation and coherency]\label{def:orientation}
The matrix $\mat{J}_X$ defined in Theorem~\ref{pro:wav:based:orientation} (\ref{def:newtensor}) is called the structure tensor of $X$. Let $\lambda_1,\lambda_2$ be its two eigenvalues. The coherency index of $X$ is defined as
\[
\chi(X)=\frac{|\lambda_2-\lambda_1|}{\lambda_1+\lambda_2}\;.
\]
An orientation $\vect{\vec n}$ is any unit eigenvector associated to the largest eigenvalue of $\mat{J}_X$.
\end{definition}

%%%%%%%%%%%%%%%%%%%%%%%%%%
\subsection{Examples\\}\label{s:examples}

We present below several examples of self-similar Gaussian fields, and every time we make explicit their structure tensor and their intrinsic orientation. 

%%%%%%%%%%
\subsubsection{Example 1: Fractional Brownian Field (FBF)\\}\label{s:wav:based:orientation:ssi:ex1}

We begin with the Fractional Brownian Field (FBF) defined in~\cite{Lind93}. This random field is the isotropic multidimensional extension of the famous Fractional Brownian Motion defined by \cite{MVN68}. 
Its harmonizable representation is:
\[
X(\vect{x})=\int_{\mathbb{R}^2} \frac{1}{\norme{\vect{\xi}}^{H+1}} ~(\expj{\ip{\vect{x}}{\vect{\xi}}}-1)\frac{\ope{\widehat{W}}(\diff \vect{\xi})}{2\pi}\;.
\]
We can easily check ($C_X\equiv \frac{1}{2\pi}$), setting $\vect{\Theta}=(\cos \theta, \sin \theta) \in \S^1$, that
\[
[\mat{J}_X]_{1,1}= \frac{1}{2\pi}\int_0^{2\pi}\cos^2 \theta\diff \theta= \frac{1}{2},
\quad [\mat{J}_X]_{1,2}=[\mat{J}_X]_{2,1}= \frac{1}{2\pi}\int_0^{2\pi}\cos \theta \sin \theta\diff \theta=0\;, 
\]
and
\[
[\mat{J}_X]_{2,2}= \frac{1}{2\pi}\int_0^{2\pi}\sin^2 \theta\diff \theta= \frac{1}{2}\;,
\]
which directly implies
\[
\mat{J}_X=\begin{pmatrix}\frac{1}{2} & 0 \\ 0 &\frac{1}{2}\end{pmatrix},\quad \chi(X)=0\;.
\]
Any unit vector is thus an orientation of the FBF, which is clearly consistent with its isotropic nature.
%%%%%%%%%%%%%%%%
\subsubsection{Example 2: Elementary fields (EF)\\}\label{s:wav:based:orientation:ssi:ex2}

 We focus here to a special case of $H$-sssi fields called {\it elementary fields} (EF), whose anisotropy function is a cone with an orientation $\alpha_0$ and a half-width $\delta>0$. It admits the following harmonizable representation
($\alpha_0\in (-\pi/2,\pi/2)$, $\delta >0$):
\begin{equation}
\label{eq:lafbf}
X_{\alpha_0,\delta}(\vect{x})=\int_{\mathbb{R}^2} (\expj{\ip{\vect{x}}{\vect{\xi}}}-1)f_X^{1/2}(\vect{\xi})\widehat{\ope{W}}(\diff \vect{\xi})\;,
\end{equation}
where 
\begin{align}
\label{eq:elementary}
&f_X(\vect{\xi})=\norme{\vect{\xi}}^{-2H-2}C_X\left(\frac{\vect{\xi}}{\norme{\vect{\xi}}}\right)\;,\\
&C_X(\vect{\Theta})=\frac{1}{4\delta}\left(\mathds{1}_{[\alpha_0-\delta,\alpha_0+\delta]}(\arg \vect{\Theta})+\mathds{1}_{[\alpha_0+\pi-\delta,\alpha_0+\pi+\delta]}(\arg \vect{\Theta})\right)\;.\label{eq:CXcone}
\end{align}
Let us compute its structure tensor $\mat{J}_X$, using the definition given in Theorem~\ref{pro:wav:based:orientation}.

\begin{remark}
The contribution of each of the portions of the cone is identical in the integrals defining the coefficients of $\mat{J}_X$. Thus, in the following computations we only consider one portion whose contribution is doubled.
\end{remark}

We start with the diagonal terms:
\[
[\mat{J}_X]_{1,1}=\int_{\vect{\Theta}\in\mathbb{S}^{1}} {\Theta}_{1}^2 ~C_X(\vect{\Theta})\diff \vect{\Theta}
= \frac{1}{2\delta} \int_{\alpha_0-\delta}^{\alpha_0+\delta}\cos^2 \theta \diff \theta
=\frac{1}{2}+\frac 1 2 \cos(2\alpha_0)\frac{\sin(2\delta)}{2\delta}\;.
\]
By the relation $\cos^2\theta+\sin^2\theta=1$, we get as well
\[
[\mat{J}_X]_{2,2}=\frac{1}{2}-\frac 1 2 \cos(2\alpha_0)\frac{\sin(2\delta)}{2\delta}\;.
\]
The last terms are computed as follows:
\[
[\mat{J}_X]_{1,2}=[\mat{J}_X]_{2,1}= \frac{1}{2\delta} \int_{\alpha_0-\delta}^{\alpha_0+\delta}\cos \theta\sin \theta \diff \theta
=\frac 1 2\sin(2\alpha_0)\frac{\sin(2\delta)}{2\delta}\;.
\]
Hence the structure tensor of the elementary field is
\[
\mat{J}_X=\begin{pmatrix}
\frac{1}{2}+\frac 1 2 \cos(2\alpha_0)\frac{\sin(2\delta)}{2\delta} & \frac 1 2\sin(2\alpha_0)\frac{\sin(2\delta)}{2\delta} \\ 
\frac 1 2\sin(2\alpha_0)\frac{\sin(2\delta)}{2\delta} & \frac{1}{2}-\frac 1 2 \cos(2\alpha_0)\frac{\sin(2\delta)}{2\delta}\end{pmatrix}\;.
\]
Remark that $\mat{J}_X$ diagonalizes as
\[
\mat{J}_X=\begin{pmatrix}
\cos \alpha_0 & -\sin \alpha_0 \\ \sin \alpha_0 & \cos \alpha_0\end{pmatrix} 
\begin{pmatrix}\frac 1 2+\frac 1 2\frac{\sin(2\delta)}{2\delta} & 0 \\ 0 & \frac 1 2-\frac 1 2\frac{\sin(2\delta)}{2\delta}\end{pmatrix}
\begin{pmatrix}\cos \alpha_0 & -\sin \alpha_0 \\ \sin \alpha_0 & \cos \alpha_0\end{pmatrix}^{\T}\;.
\]
Denoting $\lambda_1\geq \lambda_2$ the two eigenvalues of $\mat{J}_X$,
\[
\lambda_1=\frac 1 2+\frac 1 2\frac{\sin(2\delta)}{2\delta} \mbox{ and }
\lambda_2=\frac 1 2-\frac 1 2\frac{\sin(2\delta)}{2\delta}\;,
\]
the coherency index of $X$ is given by
\[
\chi(X)=\frac{\lambda_1-\lambda_2}{\lambda_1+\lambda_2}=\frac{\sin(2\delta)}{2\delta}\;.
\]
An orientation of the elementary field $X_{\alpha_0,\delta}$ being a unit eigenvector associated with $\lambda_1$, we obtain  
\[
\vect{\vec n}=\binom{\cos \alpha_0}{\sin \alpha_0}\;.
\]

\begin{remark}
This is in accordance with what we observe when performing simulations. For $\delta=\pi/2$ we recover the classical FBF which is isotropic and we get $\chi(X)=\sin(\pi)/\pi=0$. Conversely, notice that $\chi(X)$ tends to $1$ when $\delta\to 0$, meaning that the coherency is strong when the cone of admissible directions is tight around the angle $\alpha_0$. Note that in the limite case $\delta\to 0$, the density function $C_X$ tends to the Dirac measure along the line $\arg \vect{\Theta}=\alpha_0$, and the structure tensor degenerates to 
\begin{equation*}
\mat{J}_X= 
\begin{pmatrix}\cos^2 \alpha_0 & \cos \alpha_0\sin \alpha_0\\  
\cos \alpha_0\sin  \alpha_0 & \sin^2 \alpha_0
\end{pmatrix}\;,
\end{equation*}
which diagonalizes as
\[
\mat{J}_X= \mat{R}_{\alpha_0}\begin{pmatrix}1 & 0\\ 0 & 0\end{pmatrix}\mat{R}_{\alpha_0}^{\T}\;.
\]
leading to the same orientation vector $(\cos \alpha_0, \sin \alpha_0)^\T$. Notice also that in the limit case, the structure tensor is not invertible anymore.
\end{remark}
%%%%%%%%%%%%%%%%%%
\subsubsection{Example 3: sum of two elementary fields\\}\label{s:wav:based:orientation:ssi:ex3}

To understand how our notion of  Riesz-based orientation can be adapted to the setting of multiple oriented random fields, we consider the following toy model ($\alpha_0, \alpha_1 \in (-\pi/2,\pi/2)$, $\delta>0$):
\[
X=X_{\alpha_0,\delta}+X_{\alpha_1,\delta}\;.
\] 
The Gaussian field $X$ is then the sum of two elementary fields of same regularity $H$ and of respective directions $\alpha_0\neq\alpha_1$ (as defined in Example 2 above).
We assume that $\delta < |\alpha_1-\alpha_0|/2$. This last condition implies in particular that $[\alpha_0-\delta,\alpha_0+\delta]\cap [\alpha_1-\delta,\alpha_1+\delta]=\varnothing$ and the spectral densities have disjoint supports, that is $X_{\alpha_0,\delta}$ and $X_{\alpha_1,\delta}$ are supposed to be independent. 

Then we have
\[
\mat{J}_X=\begin{pmatrix}
1+\frac 1 2\frac{\sin(2\delta)}{2\delta}( \cos(2\alpha_0)+\cos(2\alpha_1)) & \frac 1 2\frac{\sin(2\delta)}{2\delta}(\sin(2\alpha_0)+\sin(2\alpha_1)) \\ \frac 1 2\frac{\sin(2\delta)}{2\delta}(\sin(2\alpha_0)+\sin(2\alpha_1)) & 1-\frac 1 2\frac{\sin(2\delta)}{2\delta}( \cos(2\alpha_0)+\cos(2\alpha_1))
\end{pmatrix}\;.
\]
As previously the matrix $\mat{J}_X$ diagonalizes as
\[
\mat{J}_X=\mat{R}_{(\alpha_0+\alpha_1)/2} \begin{pmatrix}
1+\frac{\sin(2\delta)}{2\delta}\cos(\alpha_0-\alpha_1) & 0 \\ 0 & 1-\frac{\sin(2\delta)}{2\delta}\cos(\alpha_0-\alpha_1)\end{pmatrix}\mat{R}_{(\alpha_0+\alpha_1)/2}^{\T}\;,
\]
where we denoted
\[
\mat{R}_{(\alpha_0+\alpha_1)/2}=\begin{pmatrix}\cos \left(\frac{\alpha_0+\alpha_1}{2}\right) & -\sin \left(\frac{\alpha_0+\alpha_1}{2}\right) \\ \sin \left(\frac{\alpha_0+\alpha_1}{2}\right) & \cos \left(\frac{\alpha_0+\alpha_1}{2}\right)\end{pmatrix}.
\]
Thus, the coherency index is
\[
\chi(X)=\frac{\sin(2\delta)}{2\delta}\cos(\alpha_0-\alpha_1)\;,
\]
which tends to $\cos(\alpha_0-\alpha_1)$ when $\delta\to 0$. One can also observe that the closer $\alpha_0$ and $\alpha_1$ are, the more coherent the random field is and then admits a dominant orientation. An orientation of $X$ is given by  
\[
\vect{\vec n}=\binom{\cos \left(\frac{\alpha_0+\alpha_1}{2}\right)}{\sin \left(\frac{\alpha_0+\alpha_1}{2}\right)}\;.
\]
We then recover a dominant orientation, related to half the angle of the two orientations. 
%%%%%%%%%%%%%%%%%%%%%%%%
\subsubsection{Example 4: deformation of an elementary field by a linear transform\\}\label{s:wav:based:orientation:ssi:ex4}

Let $\mat{L}$ be an invertible $2\times 2$ matrix and $\alpha_0\in (-\pi/2,\pi/2)$, $\delta>0$. Set
\begin{equation}\label{e:XL}
X_{\mat{L}}(\vect{x})=X_{\alpha_0,\delta}(\mat{L}^{-1}\vect{x})\;.
\end{equation}
Thanks to Proposition~\ref{pro:linear}, we have an explicit expression for the spectral density of $X_{\mat{L}}$:
\begin{align*}
f_{X_{\mat{L}}}(\vect{\xi})
&=\frac{\lvert \mathrm{det}(\mat{L})\rvert}{2\delta} ~ \lVert \mat{L}^{\T} \vect{\xi}\rVert^{-2H-2}\mathds{1}_{[\alpha_0-\delta, \alpha_0+\delta]}(\mathrm{arg}(\mat{L}^{\T}\vect{\xi}))\;,\\
&=\frac{\lvert \mathrm{det}(\mat{L})\rvert}{2\delta} ~\lVert \mat{L}^{\T} \vect{\xi}\rVert ^{-2H-2}\mathds{1}_{[\tan(\alpha_0-\delta),\tan(\alpha_0+\delta)]}((\mat{L}^{\T}\vect{\xi})_2/(\mat{L}^{\T}\vect{\xi})_1)\;.
\end{align*}
Since the matrix $\mat{L}$ is invertible, it admits a Singular Value Decomposition $\mat{L}=\mat{U} \mat{\Delta} \mat{V}^{\T}$ where $\mat{U},\mat{V}$ are two orthogonal matrices and $\mat{\Delta}$ a diagonal matrix with nonnegative eigenvalues. One can then deduce the general case of an invertible matrix $\mat{L}$ from three specific ones: $\mat{L}\in O^+_2(\mathbb{R})$, $\mat{L}\in O^-_2(\mathbb{R})$ and $\mat{L}$ diagonal with nonnegative eigenvalues. Before deriving the general form of an orientation vector, we will consider each term of the SVD.

\medskip

\noindent {\bf(i)} We first consider the case where $\mat{L}$ is an orthogonal matrix of the form 
\[
\mat{L}=\mathbf{R}_{\theta_0}=\begin{pmatrix}\cos \theta_0&-\sin \theta_0 \\ \sin \theta_0 &\cos \theta_0\end{pmatrix}\;,
\]
one has 
\[
f_{X_{\mat{L}}}(\vect{\xi})=\frac{1}{2\delta}~\lVert \vect{\xi}\rVert^{-2H-2}\mathds{1}_{[\alpha_0+\theta_0-\delta,\alpha_0+\theta_0+\delta]}(\mathrm{arg}~\vect{\xi})\;,
\]
which implies that one can choose as orientation for $X_{\mat{L}}$ the unit vector
\begin{equation}\label{e:orth1}
\vect{\vec n}_{\mat{L}}=\binom{\cos \left(\alpha_0+\theta_0\right)}{\sin \left(\alpha_0+\theta_0\right)}=\mat{L} \binom{\cos \alpha_0}{\sin\alpha_0}=(\mat{L}^{-1})^\T \binom{\cos \alpha_0}{\sin \alpha_0}\;,
\end{equation}
since any orthogonal matrix equals the transpose of its inverse.

\medskip

\noindent {\bf(ii)} We now deal with the case where $\mat{L}$ is an orthogonal matrix of the form 
\[
\mat{L}=\begin{pmatrix}\cos \theta_0&\sin \theta_0 \\ \sin \theta_0&-\cos \theta_0 \end{pmatrix}=\begin{pmatrix}\cos \theta_0&-\sin \theta_0 \\ \sin \theta_0&\cos \theta_0 \end{pmatrix}\times\begin{pmatrix}1&0 \\ 0&-1 \end{pmatrix}\;.
\] 
One has 
\[
f_{X_{\mat{L}}}(\vect{\xi})=\frac{1}{2\delta} \lVert \vect{\xi}\rVert^{-2H-2}\mathds{1}_{[\theta_0-\alpha_0-\delta,\theta_0-\alpha_0+\delta]}(\mathrm{arg}\, \vect{\xi} )\;,
\]
which implies that one can choose as orientation for $X_{\mat{L}}$ the unit vector
\begin{equation}\label{e:orth2}
\vect{\vec n}_{\mat{L}}=\binom{\cos \left(\theta_0-\alpha_0\right)}{\sin \left(\theta_0-\alpha_0\right)}=\mat{L}\binom{\cos \alpha_0}{\sin \alpha_0}=(\mat{L}^{-1})^\T \binom{\cos \alpha_0}{\sin \alpha_0}\;,
\end{equation}
since as above any orthogonal matrix equals the transpose of its inverse.
\medskip

\noindent {\bf(iii)} We finally deal with the case where $\mat{L}$ is a diagonal matrix  $\mat{L}=\begin{pmatrix}\lambda_1&0\\0&\lambda_2\end{pmatrix}$, with $\lambda_1,\lambda_2 > 0$. In this case observe that the condition
\[
\tan(\alpha_0-\delta)<\frac{(\mat{L}^{\T}\vect{\xi})_2}{(\mat{L}^{\T}\vect{\xi})_1}<\tan(\alpha_0+\delta)\;,
\]
is equivalent to $\frac{\lambda_1}{\lambda_2}\tan(\alpha_0-\delta)<\frac{\xi_2}{\xi_1}<\frac{\lambda_1}{\lambda_2}\tan(\alpha_0+\delta)$, that is to
\[
\underline{\delta}_\mat{\Delta}<\arg\, \vect{\xi}<\overline{\delta}_\mat{\Delta}\;,
\]
with $\underline{\delta}_\mat{\Delta}=\mathrm{arctan}(\frac{\lambda_1}{\lambda_2}\tan(\alpha_0-\delta))$ and $\overline{\delta}_\mat{\Delta}=\mathrm{arctan}(\frac{\lambda_1}{\lambda_2}\tan(\alpha_0+\delta))$. Hence,
\begin{align*}
&f_{X_{\mat{L}}}(\vect{\xi})=\frac{\lvert \mathrm{det}(\mat{L})\rvert}{2\delta}~ \lVert \mat{L}^{\T} \vect{\xi}\rVert^{-2H-2} \mathds{1}_{[\underline{\delta}_\mat{\Delta},\overline{\delta}_\mat{\Delta}]}(\arg \vect{\xi})\;,\\
&C_{X_{\mat{L}}}(\vect{\Theta})=\frac{\lvert \mathrm{det}(\mat{L})\rvert}{2\delta} \lVert \mat{L}^{\T} \vect{\Theta}\rVert^{-2H-2} \mathds{1}_{[\underline{\delta}_\mat{\Delta},\overline{\delta}_\mat{\Delta}]}(\arg \vect{\Theta})\;.
\end{align*}
Now, recalling that:
$$[\mat{J}_{X_{\mat{L}}}]_{\ell_1,\ell_2}=\int_{\vect{\Theta}\in\mathbb{S}^{1}} {\Theta}_{\ell_1}{\Theta}_{\ell_2}~C_{X_{\mat{L}}}(\vect{\Theta})\dif \vect{\Theta}~, $$
we obtain
\[
\frac{[\mat{J}_{X_{\mat{L}}}]_{1,1}}{\lvert \mathrm{det}(\mat{L})\rvert}= \frac{1}{2\delta}
\int_{\underline{\delta}_\mat{\Delta}}^{\overline{\delta}_\mat{\Delta}} \cos^2 \theta~C_{X_{\mat{L}}}(\cos \theta,\sin \theta)\diff \theta=  \frac{1}{2\delta}
\int_{\underline{\delta}_\mat{\Delta}}^{\overline{\delta}_\mat{\Delta}} \frac{\cos^2 \theta}{(\lambda_1^2 \cos^2 \theta+\lambda_2^2 \sin^2 \theta)^{H+1}}\diff \theta\;,
\]
\[
\frac{[\mat{J}_{X_{\mat{L}}}]_{2,2}}{\lvert \mathrm{det}(\mat{L})\rvert}= 
 \frac{1}{2\delta} \int_{\underline{\delta}_\mat{\Delta}}^{\overline{\delta}_\mat{\Delta}}\sin^2 \theta~ C_{X_{\mat{L}}}(\cos \theta,\sin \theta)\diff \theta= \frac{1}{2\delta}
 \int_{\underline{\delta}_\mat{\Delta}}^{\overline{\delta}_\mat{\Delta}} \frac{\sin^2 \theta}{(\lambda_1^2 \cos^2 \theta+\lambda_2^2 \sin^2 \theta)^{H+1}}\diff \theta\;,
\]
\[
\frac{[\mat{J}_{X_{\mat{L}}}]_{1,2}}{\lvert \mathrm{det}(\mat{L})\rvert}= \frac{1}{2\delta}
\int_{\underline{\delta}_\mat{\Delta}}^{\overline{\delta}_\mat{\Delta}} \cos \theta\sin \theta~C_{X_{\mat{L}}}(\cos \theta,\sin \theta)\diff \theta=  \frac{1}{2\delta}\int_{\underline{\delta}_\mat{\Delta}}^{\overline{\delta}_\mat{\Delta}} \frac{\cos\theta\sin \theta}{(\lambda_1^2 \cos^2 \theta+\lambda_2^2 \sin^2 \theta)^{H+1}}\diff \theta\;.
\]

\medskip

\noindent
Now, let us define $(u_1(\theta),u_2(\theta))=(\cos \theta,\sin \theta)$ and let us introduce the functions
\begin{align*}
 f_{\ell_1,\ell_2}&:\theta\mapsto u_{\ell_1}(\theta)u_{\ell_2}(\theta)~(\lambda_1^2u_1(\theta)^2+\lambda_2^2u_2(\theta)^2)^{-H-1}\;,
& F_{\ell_1,\ell_2} : x\mapsto   \int_0^x f_{\ell_1,\ell_2}(\theta)\diff \theta,  \\
  \nu &: \alpha\mapsto \arctan\left(\frac{\lambda_1}{\lambda_2}\tan \alpha\right),
  & G_{\ell_1,\ell_2} : \alpha\mapsto F_{\ell_1,\ell_2}(\nu (\alpha))\;. 
 \end{align*}
 Each term of the structure tensor writes
\[
\frac{[\mat{J}_{X_{\mat{L}}}]_{\ell_1,\ell_2}}{\lvert \mathrm{det}(\mat{L})\rvert}=\frac{G_{\ell_1,\ell_2}(\alpha_0+\delta)-G_{\ell_1,\ell_2}(\alpha_0-\delta)}{2\delta}\;.
\]
When the parameter $\delta$ is small, we have
\begin{align*}
\frac{G_{\ell_1,\ell_2}(\alpha_0+\delta)-G_{\ell_1,\ell_2}(\alpha_0-\delta)}{2\delta}&=G_{\ell_1,\ell_2}'(\alpha_0) + \frac{\delta^2}{12}G_{\ell_1,\ell_2}'''(\alpha_0) + O(\delta^4)\;
,\\
&=\nu '(\alpha_0)~F_{\ell_1,\ell_2}'(\nu (\alpha_0))+ O(\delta^2)\;,\\
&=\nu '(\alpha_0)~f_{\ell_1,\ell_2}(\nu (\alpha_0))+ O(\delta^2)\;.
\end{align*}
Hence,
\[
\frac{G_{\ell_1,\ell_2}(\alpha_0+\delta)-G_{\ell_1,\ell_2}(\alpha_0-\delta)}{2\delta}
= \frac{\lambda_1\lambda_2}{\lambda_2^2\cos^2\alpha_0+\lambda_1^2\sin^2\alpha_0}\times 
\frac{u_{\ell_1}(\nu(\alpha_0))~u_{\ell_2}(\nu(\alpha_0))}{(\lambda_1^2\cos^2(\nu(\alpha_0))+\lambda_2^2\sin^2(\nu(\alpha_0)))^{H+1}}+ O(\delta^2)\;.
\]
Let us define $C_{H,\lambda_1,\lambda_2,\alpha_0}=\lambda_1^2 \lambda_2^2 (\lambda_2^2\cos^2\alpha_0+\lambda_1^2\sin^2\alpha_0)^{-1}(\lambda_1^2\cos^2(\nu(\alpha_0))+\lambda_2^2\sin^2(\nu(\alpha_0)))^{-H-1}$.
Then, one has for small $\delta$,
\begin{equation*}
\mat{J}_{X_{\mat{L}}} = C_{H,\lambda_1,\lambda_2,\alpha_0}~ 
\begin{pmatrix}\cos^2(\nu (\alpha_0)) & \cos (\nu (\alpha_0))\sin (\nu (\alpha_0))\\  \cos (\nu (\alpha_0))\sin (\nu (\alpha_0)) & \sin^2(\nu (\alpha_0))\end{pmatrix}+ O(\delta^2)\;,
\end{equation*}
which can be written as
\[
\mat{J}_{X_{\mat{L}}} = C_{H,\lambda_1,\lambda_2,\alpha_0} \mat{R}_{\nu(\alpha_0)}\begin{pmatrix}1 & 0\\ 0 & 0\end{pmatrix}\mat{R}_{\nu(\alpha_0)}^{\T}+ O(\delta^2)\;.
\]
Therefore $\vect{\vec n}_{\mat{L}}=(\cos \nu(\alpha_0),\sin \nu (\alpha_0))^{\T}$ can be viewed as an approximate eigenvector of $\mat{J}_{X_{\mat{L}}}$ associated to its largest eigenvalue, and then an orientation of $X_{\mat{L}}$.
Finally, we remark that
\[
\nu (\alpha_0)=\arctan \left( \frac{\lambda_1 \sin \alpha_0}{\lambda_2\cos \alpha_0}\right)=\arg\left[ \begin{pmatrix} \lambda_2 \cos \alpha_0\\ \lambda_1 \sin \alpha_0\end{pmatrix}\right]=\arg\left[ \begin{pmatrix}\lambda_2 & 0\\ 0 & \lambda_1\end{pmatrix}\times \binom{\cos \alpha_0}{\sin\alpha_0}\right]\;.
\]
Consequently, an approximate (up to $\delta^2$) orientation of $X_{\mat{L}}$ is in this case
\[
\vect{\vec n}_{\mat{L}}=\frac{\begin{pmatrix}\lambda_2 & 0\\ 0 & \lambda_1\end{pmatrix}\vect{\vec n}}{\left \lVert \begin{pmatrix}\lambda_2 & 0\\ 0 & \lambda_1\end{pmatrix}\vect{\vec n}\right \rVert}\mbox{ with }\vect{\vec n}=\binom{\cos \alpha_0}{\sin\alpha_0}\;.
\]
Observe that $\begin{pmatrix}\lambda_2 & 0\\ 0 & \lambda_1\end{pmatrix}$ is the adjugate matrix of $\mat{L}$. Then, dividing the numerator and denominator of the last equation by $\mathrm{det}(\mat{L})=\lambda_1 \lambda_2$, we get
\begin{equation}\label{e:diag}
\vect{\vec n}_{\mat{L}}=\frac{\mat{L}^{-1} \vect{\vec n}}{\lVert \mat{L}^{-1} \vect{\vec n}\rVert}=\frac{(\mat{L}^{-1})^\T \vect{\vec n}}{\lVert (\mat{L}^{-1})^\T \vect{\vec n}\rVert}\;,
\end{equation}
since the diagonal matrix $\mat{L}^{-1}$ equals its transpose.

We now gather~(\ref{e:orth1}), (\ref{e:orth2}) and~(\ref{e:diag}): using the existence of the SVD for every matrix, we deduce the following proposition.
\begin{proposition}\label{prop:orientation_vector}
Let $\mat{L}$ be an invertible $2\times 2$ matrix and $X_{\mat{L}}$ the Gaussian field defined by~(\ref{e:XL}). Set $\vect{\vec n}=\binom{\cos \alpha_0}{\sin\alpha_0}$ the orientation vector of $X$. Then the unit vector 
\[
\vect{\vec n}_{\mat{L}}=\frac{(\mat{L}^{-1})^{\T}\vect{\vec n}}{\lVert (\mat{L}^{-1})^{\T} \vect{\vec n}\rVert}\;,
\]
is an approximate (up to $\delta^2$) orientation vector of $X_{\mat{L}}$.
\end{proposition}
\begin{proof} Let $\mat{L}=\mat{U}\mat{\Delta} \mat{V^T}$ be the SVD singular decomposition of $\mat{L}$,  with $\mat{U}, \mat{V}\in O_2(\mathbb{R})$ and $\mat{\Delta}$ diagonal with nonnegative eigenvalues.\\
 $X_{\mat{L}}(\vect{x})=X(\mat{V}^{-\T}\mat{\Delta}^{-1}\mat{U}^{-1}\vect{x})$ be the Gaussian field defined by~(\ref{e:XL}). Let decompose the three operations as follows	:
 \[
X_{\mat{L}}=\underbrace{\underbrace{X\circ (\mat{V}^\T)^{-1}}_{X_1}\circ \mat{\Delta}^{-1}}_{X_2}\circ \mat{U}^{-1}\;.
\]
Then, since $\mat{V}^\T$ is an orthogonal matrix, we have from (i) and (ii) that the unit orientation vector of $X_1$ is 
\[
\vect{\vec n}_{1} = ((\mat{V}^\T)^{-1})^{\T} \vect{\vec n}\;.
\]
Now from (iii), the unit orientation vector (up to $\delta^2$) of $X_2=X_1\circ \mat{\Delta}^{-1}$ is
\[
\vect{\vec n}_{2} = \frac{(\mat{\Delta}^{-1})^{\T}\vect{\vec n}_{1}}{\lVert (\mat{\Delta}^{-1})^{\T}\vect{\vec n}_{1}\rVert} = \frac{(\mat{\Delta}^{-1})^{\T}((\mat{V}^\T)^{-1})^{\T} \vect{\vec n}}{\lVert (\mat{\Delta}^{-1})^{\T}((\mat{V}^\T)^{-1})^{\T} \vect{\vec n}\rVert}\;.
\]
Finally, from (i) and (ii) again, the unit orientation vector (up to $\delta^2$) of $X_{\mat{L}}=X_2\circ \mat{U}^{-1}$ is
\begin{align*}
\vect{\vec n}_{\mat{L}}
&= \frac{( \mat{U}^{-1})^{\T} \vect{\vec n}_{2}}{\lVert ( \mat{U}^{-1})^{\T} \vect{\vec n}_{2}\rVert}\;,\\
&=\frac{ (\mat{U}^{-1})^{\T} (\mat{\Delta}^{-1})^{\T}((\mat{V}^\T)^{-1})^{\T} \vect{\vec n}}{\lVert (\mat{U}^{-1})^{\T} (\mat{\Delta}^{-1})^{\T}((\mat{V}^\T)^{-1})^{\T} \vect{\vec n}\rVert}\;,\\
&=\frac{ ((\mat{V}^\T)^{-1}\mat{\Delta}^{-1}\mat{U}^{-1})^{\T} \vect{\vec n}}{\lVert ((\mat{V}^\T)^{-1}\mat{\Delta}^{-1}\mat{U}^{-1})^{\T} \vect{\vec n}\rVert}=\frac{ (\mat{U}\mat{\Delta}\mat{V}^\T)^{\T} \vect{\vec n}}{\lVert (\mat{U}\mat{\Delta}\mat{V}^\T)^{\T} \vect{\vec n} \rVert}=\frac{(\mat{L}^{-1})^{\T}\vect{\vec n}}{\lVert (\mat{L}^{-1})^{\T} \vect{\vec n}\rVert}\;.
\end{align*}

\end{proof}
%%%%%%%%%%%%%%%

\section{Riesz-based orientation of localizable Gaussian fields}\label{s:wav:based:orientation:loc}

We now extend the notion of intrinsic orientation, defined for self-similar random fields, to a much more general setting, that of localizable Gaussian fields. This will be the purpose of Section~\ref{s:def-orientation-localizable}. In Section~\ref{s:ex-orientation-localizable}, we will apply it to two classes of model with prescribed orientation. First of all, let us recall in Section~\ref{s:localizable} the definition of localizable Gaussian fields.
\subsection{Localizable Gaussian fields}
\label{s:localizable}

We first recall, following~\cite{BJR97,Falc02,Falc03}, the definition of $H$--localizable Gaussian fields.
\begin{definition}[Localizable Gaussian field]
Let $H\in (0,1)$. We say that the random field $Y = \{Y (\vect{x}),\, \vect{x}\in\mathbb{R}^2\}$ is $H$--localizable at $\vect{x_0}\in\mathbb{R}^2$ with \textit{tangent field} (or \textit{local form}) the non-trivial random field
$Y_{\vect{x_0}} = \{Y_{\vect{x_0}}(\vect{x}),\,\vect{x} \in\mathbb{R}^2\}$ if
\begin{equation}\label{e:def:localizable}
\left\{\frac{Y(\vect{x_0} + \rho\vect{h})-Y(\vect{x_0})}{\rho^H}\right\}_{\vect{h}\in \mathbb{R}^2} \overset{d}{\longrightarrow}\left\{Y_{\vect{x_0}}(\vect{h})\right\}_{\vect{h}\in \mathbb{R}^2}\;,
\end{equation}
as $\rho\to 0$, where $\overset{d}{\rightarrow}$ means convergence in distribution, that is the weak convergence for stochastic processes (see \cite{Bill68}). \\
A random field $Y= \{Y (\vect{x}),\, \vect{x}\in\mathbb{R}^2\}$ is said to be localizable if for all $\vect{x}\in\mathbb{R}^2$ it is $H$--localizable for some $H\in (0,1)$.
\end{definition}
In Theorem~3.9 and Corollary~3.10 of \cite{Falc02}, Falconer proved the
following result that we state in the Gaussian case. It enables
to describe the whole class of possible tangent fields of a Gaussian field with continuous sample paths.
\begin{theorem}\label{th:tangent:fields}
Let $X$ be a localizable Gaussian field with continuous sample paths. For almost all $\vect{x_0}$ in $\mathbb{R}^{2}$
the tangent field $Y_{\vect{x_0}}$ of $X$ at $\vect{x_0}$ has stationary increments and is self-similar, that is for some $H\in (0,1)$ and for all $\rho\geq 0$,
\begin{equation}
\{Y_{\vect{x_0}}(\rho \vect{x}),\,\vect{x}\in\mathbb{R}^2\}\overset{(fdd)}{=}\{\rho^H Y_{\vect{x_0}}(\vect{x}),\,\vect{x}\in\mathbb{R}^2\}\;.
\end{equation}
\end{theorem}
In short, a Gaussian field with continuous sample paths will have at a.e. point, a  ``fractal''  tangent field behaving like a FBF. \\
We now illustrate this notion considering a classical example of Gaussian field with prescribed tangent field: the Multifractional Brownian Field defined in the unidimensional setting in~\cite{LVPel95}, and in the multivariate case in~\cite{BJR97,Herb06}. Such field is localizable at each point, with a fractional Brownian Field for tangent field.
\begin{example}[Multifractional Brownian Field]
\label{ex:MBF}
Let $h:\mathbb{R}^2\to (0,1)$ be a continuously differentiable function whose range is supposed to be a compact interval $[\alpha,\beta]\subset (0,1)$. The Multifractional Brownian Field (MBF) with multifractional function $h$, is the Gaussian field defined by its harmonisable representation as follows
\begin{equation}\label{e:multi}
X_h(\vect{x})=\int_{\mathbb{R}^2}\frac{\expj{\ip{\vect{x}}{\vect{\xi}}}-1}{\lVert \vect{\xi}\rVert^{h(\vect{x})+1}}\ope{\widehat{W}}(\diff \vect{\xi})\;.
\end{equation}
\end{example}

%%%%%%%%%%
\subsection{Tensor structure and orientation of localizable Gaussian fields\\}\label{s:def-orientation-localizable}

The results of Section~\ref{s:self-similar} together with Theorem~\ref{th:tangent:fields} of section~\ref{s:localizable}, will allow us to define the Riesz-based orientation of any $H$--localizable Gaussian field $X$ almost everywhere.

\begin{definition}[Localizable field orientations]\label{def:localizable_orientation}
Let $X$ be a Gaussian field with continuous sample paths. Assume that $X$ is localizable at the point $\vect{x_0}$, with tangent field $Y_{\vect{x_0}}$, and that $Y_{\vect{x_0}}$ is a self-similar Gaussian field with stationary increments. One then defines:
\begin{itemize}
\item The \underline{local anisotropy function} $C_{\vect{x_0}}$ at $\vect{x_0}$ of the localizable Gaussian field $X$ is the anisotropy function of its tangent field $Y_{\vect{x_0}}$.
\item The  \underline{local structure tensor} $\mat{J}_{\vect{x_0}}$ at $\vect{x_0}$ of the localizable Gaussian field $X$ is the structure tensor of its tangent field $Y_{\vect{x_0}}$.
\item A  \underline{local orientation} at $\vect{x_0}$ of the localizable Gaussian field $X$ is any orientation of its tangent field $Y_{\vect{x_0}}$.
\end{itemize}
\end{definition}
In view of these definitions and of Theorem~\ref{th:tangent:fields}, we deduce that any localizable Gaussian field $X$ admits a local orientation at almost every point $\vect{x_0}\in\mathbb{R}^2$.

\begin{example}[Local structure tensor and orientation of a MBF]
\label{pro:ex:multi}
The Multifractional Brownian Field $X_h$ \eqref{e:multi} admits at each point a structure tensor proportional to the identity matrix. In particular, any unit vector is an orientation of $X_h$.
Indeed, the tangent field $Y_{\vect{x_0}}$ of the MBF at point $\vect{x_0}$ is a FBF of Hurst index $h(\vect{x}_0)$, whose structure tensor has been determined in example \ref{s:wav:based:orientation:ssi:ex1}.
\end{example}

%%%%%%%%%%%%%%%%%%
\subsection{Two new models of localizable Gaussian fields with prescribed orientation\\}\label{s:ex-orientation-localizable}

In this section, we will extend our previous works \cite{PCCP14,PCCP15} and define two classes of Gaussian fields with prescribed orientation. The details about numerical aspects and synthesis of the model, as well as comparison between them, will be detailed in the companion paper \cite{polisano2017simulation}. These two models will be derived from two general classes: Generalized Anisotropic Fractional Brownian Fields (GAFBF) and 
%reducible stationary Gaussian fields 
Warped  Anisotropic Fractional Brownian Fields (WAFBF)  that we describe in Sections~\ref{s:CFBA} and \ref{s:perrin} respectively.

%%%%%%%%%%%%%%
\subsubsection{First model: Generalized Anisotropic Fractional Brownian Fields (GAFBF)\\}\label{s:CFBA}

We introduce below the definition of Generalized Anisotropic Fractional Brownian Fields (GAFBF) which generalizes the notion of Locally Anisotropic Brownian Fields (LAFBF) introduced in \cite{PCCP14},  and whose simulation will be studied in \cite{polisano2017simulation}. 

Our Gaussian field will be defined from two functions $h$ from $\mathbb{R}^2$ to $[0,1]$ and $C$ from $\mathbb{R}^2\times \mathbb{R}^2$ to $\mathbb{R_+}$ satisfying the following set of assumptions:\\

\noindent{\bf Assumptions~}($\mathcal{H}$)
\,
\begin{itemize}
\item $h$ is a $\beta$--H\"older function, such that  $\displaystyle a=\inf_{\vect{x}\in \mathbb{R}^2}h(\vect{x}) > 0$, $\displaystyle b=\sup_{\vect{x}\in \mathbb{R}^2}h(\vect{x})<1$ and $b<\beta\leq 1$.
\item $(\vect{x},\vect{\xi})\mapsto C(\vect{x},\vect{\xi})$ is bounded, that is $\forall (\vect{x},\vect{\xi})\in \mathbb{R}^2\times \mathbb{R}^2,\, C(\vect{x},\vect{\xi})\leqslant M$.
\item $\vect{\xi}\mapsto C(\vect{x},\vect{\xi})$ is even and homogeneous of degree 0: $\forall \rho>0$, $C(\vect{x},\rho\vect{\xi})=C(\vect{x},\vect{\xi})$.
\item $\vect{x}\mapsto C(\vect{x},\vect{\xi})$ is continuous and satisfies: there exists some $\eta$, with $b<\eta\leq 1$ such that
\begin{equation}
\label{eq:homogene}
 \forall \vect{x}\in \mathbb{R}^2,\quad \sup_{\vect{z}\in B(\vect{0},1)} \|\vect{z}\|^{-2\eta}\int_{\mathbb{S}^1}\left[C(\vect{x}+\vect{z},\vect{\Theta})-C(\vect{x},\vect{\Theta})\right]^2\dif \vect{\Theta}\leq A_{\vect{x}}<\infty \;.
\end{equation}
Morever $\vect{x}\mapsto A_{\vect{x}}$ is bounded on any compact set of $\mathbb{R}^2$.
\end{itemize}

\noindent We now define our model, the Generalized Anisotropic Fractional Brownian Field.
\begin{definition}{Generalized Anisotropic Fractional Brownian Fields (GAFBF)}\\
Let us consider $h:\mathbb{R}^2\to [0,1]$ and $C$ satisfying Assumptions~($\mathcal{H}$). We then define the GAFBF as the following Gaussian field generalizing~\cite{PCCP14,PCCP15} by
\begin{equation}\label{eq:gafbf}
X(\vect{x})\overset{\mathrm{def}}{=}\int_{\mathbb{R}^2}(\expj{\ip{\vect{x}}{\vect{\xi}}}-1)\frac{C(\vect{x},\vect{\xi})}{\lVert \vect{\xi} \rVert^{h(\vect{x})+1}}\ope{\widehat{W}}(\diff \vect{\xi})\;.
\end{equation}
\end{definition}

The main properties of the GAFBF $X$ are summarized in the following propositions.
\begin{theorem}\label{theo:local:orientation:genCFBA}
The GAFBF $X$ \eqref{eq:gafbf} admits at any point $\vect{x_0}\in \mathbb{R}^2$, a tangent field $Y_{\vect{x_0}}$  given by
\begin{equation}\label{eq:Yx0}
Y_{\vect{x_0}}(\vect{x})=\int_{\mathbb{R}^2}(\expj{\ip{\vect{x}}{\vect{\xi}}}-1)f^{1/2}(\vect{x_0},\vect{\xi})\widehat{\ope{W}}(\diff \vect{\xi})=\int_{\mathbb{R}^2}(\expj{\ip{\vect{x}}{\vect{\xi}}}-1)\frac{C(\vect{x_0},\vect{\xi})}{\norme{\vect{\xi}}^{h(\vect{x_0})+1}}\widehat{\ope{W}}(\diff \vect{\xi})\;.\end{equation}
In particular, for each point $\vect{x_0}$, the local anisotropy function of the Gaussian field $X$ at $\vect{x_0}$ is 
\[
C_{\vect{x}_0}:\vect{\Theta}\mapsto C(\vect{x_0},\vect{\Theta})^2\;.
\]
\end{theorem}
\begin{proof}
Theorem~\ref{theo:local:orientation:genCFBA} is proven in Section~\ref{s:proof:local:orientation}.
\end{proof}

\begin{example} \label{ex:LAFBF}
We now derive our first example of Gaussian field with prescribed orientation. The LAFBF introduced in \cite{PCCP14,PCCP15}  is a particular case of GAFBF where the function $C$ is a localized version of the cone \eqref{eq:CXcone} with constant half-width $\delta>0$ and whose orientation may vary spatially, that is where $\alpha : \R^2 \to (-\pi/2,\pi/2)$ is now a continuously differentiable function, which is $2\eta$--Holderian with $b<\eta\leq 1/2$. One can verify that the so-called function $C$ satisfies assumptions $(\mathcal{H})$.  Applying Theorem~\ref{theo:local:orientation:genCFBA}, the corresponding localizable Gaussian field defined by formula~(\ref{eq:gafbf}) admits a tangent field at any points $\vect{x}_0\in \mathbb{R}^2$ and the following cone as local anisotropy:
\begin{equation}
\label{eq:C}
C_{\vect{x}_0}(\vect{\Theta})=\frac{1}{4\delta}\left(\mathds{1}_{[\alpha(\vect{x_0})-\delta,\alpha(\vect{x_0})+\delta]}(\arg \vect{\Theta})+\mathds{1}_{[\alpha(\vect{x_0})+\pi-\delta,\alpha(\vect{x_0})+\pi+\delta]}(\arg \vect{\Theta})\right)\;.
%C(\vect{x_0},\vect{\xi})=\frac{1}{\sqrt{2\delta}}\mathds{1}_{[-\delta,\delta]}(\arg(\vect{\xi})-\alpha(\vect{x_0}))\;.
\end{equation}
\end{example}
The above example shows that the tangent field of a LAFBF is actually an elementary field. Then, using the results of Example \ref{s:wav:based:orientation:ssi:ex2} and Definition \ref{def:localizable_orientation}, we immediately deduce the following proposition:
\begin{proposition}\label{prop:lafbf_orientation}
The LAFBF defined in the Example \ref{ex:LAFBF}  admits at each point $\vect{x_0}$ an approximate (up to $\delta^2$) local orientation vector given by
\[
\vect{\vec n}=\begin{pmatrix}\cos \alpha(\vect{x_0})\\ \sin \alpha(\vect{x_0})\end{pmatrix}\;.
\]
\end{proposition} 
%The H\"older condition relying on $\alpha$ imposes a tradeoff between the rugosity variations of the texture governed by $h$, and the variations of the orientation governed by $\alpha$. This restriction prevent the orientation to grow too rapidly, otherwise we would observe some line artefacts in numerical simulations (see \cite{polisano2017simulation} for details). To overcome this drawback, we define below a second model based on the deformation of a $H$-sssi field.

%%%%%%%%%
\subsubsection{Warped  Anisotropic Fractional Brownian Fields (WAFBF)\\}\label{s:perrin}

We now consider a second model, satisfying similar properties, in the same spirit as the approach developed in~\cite{perrin1999reducing,Per00,guyon2000identification} but in the case where the warped Gaussian field is a $H$--self-similar Gaussian field with stationary increments.
\begin{definition}[Warped  Anisotropic Fractional Brownian Fields]
\label{def:wafbf}
Let $X$ be a $H$-sssi field, with anisotropy function $C_X$, as explicit in Proposition \ref{pro:spectral-density-ss}. 
Let $\ope{\Phi} : \mathbb{R}^2 \to \mathbb{R}^2$ be a continuously differentiable function. The Warped  Anisotropic Fractional Brownian Field (WAFBF)  $Z_{\ope{\Phi},X}$ is defined as the deformation of the elementary field $X$:
\begin{equation}\label{e:model-perrin}
Z_{\ope{\Phi},X}(\vect{x})=X(\ope{\Phi}(\vect{x}))\;.
\end{equation}
\end{definition}
The aim of this section is to study the local properties of such Gaussian fields.
\begin{proposition}\label{pro:ex2}
The Gaussian field $Z_{\ope{\Phi},X}$ defined by~(\ref{e:model-perrin}) is localizable at any point $\vect{x_0}\in\mathbb{R}^2$, with tangent field $Y_{\vect{x_0}}$ defined as
 \begin{equation}
 \label{eq:tfZ}
 Y_{\vect{x_0}}(\vect{x})=X(\ope{D}\ope{\Phi}(\vect{x_0})~ \vect{x})\;, \quad \forall  \vect{x}\in\R^2\;,
 \end{equation}
 where $\ope{D}\ope{\Phi}(\vect{x_0})$ is the Jacobian matrix of $\ope{\Phi}$ at point $\vect{x_0}$.
\end{proposition}
\begin{proof}
Proposition~\ref{pro:ex2} is proven in Section~\ref{s:proof:pro:ex2}.
\end{proof}

\begin{proposition}\label{orientation_deformation}
Let $Z_{\ope{\Phi},X}$ be the WAFBF defined in Definition \ref{def:wafbf}, from an elementary field $X=X_{\alpha_0,\delta}$ defined by \eqref{eq:lafbf}-\eqref{eq:elementary}-\eqref{eq:CXcone} and whose orientation, computed in section \ref{s:wav:based:orientation:ssi:ex2}, is given by the unit vector $\vect{\vec n}=(\cos \alpha_0, \sin \alpha_0)$. In addition, we assume that the $C^1$-differentiable deformation ${\ope{\Phi}}$ is a diffeomorphism on an open set $U\subset \R^2$.

Then, at each point $\vect{x_0}\in U$,  the WAFBF $Z_{\ope{\Phi},X}$ admits an approximate (up to $\delta^2$) local orientation given by
\[
\vect{\vec n}_Z(\vect{x_0})=\frac{\ope{D}\ope{\Phi}(\vect{x_0})^{\T}\vect{\vec n}}{\|\ope{D}\ope{\Phi}(\vect{x_0})^{\T}\vect{\vec n}\|}\;.
\]
\end{proposition}
\begin{proof}
According to Definition \ref{def:localizable_orientation}, the local orientation of $Z_{\ope{\Phi},X}$ at $\vect{x_0}\in U$ is given by the one of its tangent field $Y_{\vect{x_0}}$. From Proposition \ref{pro:ex2}, $Y_{\vect{x_0}} (\vect{x})= X(\ope{D}\ope{\Phi}(\vect{x_0})~ \vect{x})$, and since ${\ope{\Phi}}$ is a diffeomorphism in a neighborhood  of $\vect{x_0}$, $\ope{D}\ope{\Phi}(\vect{x_0})$ is invertible. Proposition \ref{prop:orientation_vector} applied to $\mat{L}^{-1}=\ope{D}\ope{\Phi}(\vect{x_0})$ directly leads to the result.
\end{proof}

%\[
%\Phi(x)=\mathbf{R}_{-\alpha(x)}\cdot x=\begin{pmatrix}\cos \alpha(x)\, x_1+\sin \alpha(x)\, x_2\\ -\sin \alpha(x)\, x_1+\cos \alpha(x)\, x_2\end{pmatrix}\equiv \begin{pmatrix}\Phi_1(x)\\ \Phi_2(x)\end{pmatrix}
%\]
%with $\alpha:\mathbb{R}^2\to \mathbb{R}$ a differentiable function on a $\mathbb{R}^2$.
%\[
%\Phi_1(x)=\int_0^{x_1}\cos(\alpha(y_1,x_2))d y_1\mbox{ and }\Phi_2(x)=\int_0^{x_2}\sin(\alpha(x_1,y_2))d y_2\;.
%\]

\begin{example}[Local rotation]
We now illustrate this result considering the case $\alpha_0=0$, then the orientation $\vect{\vec n}=(\cos \alpha_0, \sin \alpha_0)$ of the elementary field $X$ is now the unit vector $\vect{\vec n}=\vect{e_1}=(1,0)^\T$. The deformation we consider is a local rotation governed by a continuously differentiable function $\vect{x}\mapsto \alpha(\vect{x})$. We have the following proposition:
\end{example}
\begin{proposition}
We consider the following warped field from the standard elementary field:
\[
Z_{\ope{\Phi},X}(\vect{x})=X_{0,\delta}(\ope{\Phi}(\vect{x}))\;,
\] 
with
\begin{equation}
\label{eq:deformation}
\ope{\Phi}(\vect{x})=\mathbf{R}_{-\alpha(\vect{x})}\vect{x}=\begin{pmatrix}\cos \alpha(\vect{x})\, x_1+\sin \alpha(\vect{x})\, x_2\\ -\sin \alpha(\vect{x})\, x_1+\cos \alpha(\vect{x})\, x_2\end{pmatrix}\equiv \begin{pmatrix}\ope{\Phi}_1(\vect{x})\\ \ope{\Phi}_2(\vect{x})\end{pmatrix}\;,
\end{equation}
where $\alpha:\mathbb{R}^2\to \mathbb{R}$ is a $C^1$ function on $\mathbb{R}^2$ such that, on an open set $U\subset \R^2$, one has: 
\begin{equation}\label{eq:conditions_alpha}
\forall \vect{x_0}\in U, ~~~~~\nabla \alpha (\vect{x_0})  \wedge \vect{x_0} = \frac{\partial \alpha}{\partial x_1}(\vect{x_0})x_{0,2}-\frac{\partial \alpha}{\partial x_2}(\vect{x_0})x_{0,1} \neq -1\;.
\end{equation}
Then, for each point $\vect{x_0}\in U$ satisfying \eqref{eq:conditions_alpha}, $Z_{\ope{\Phi},X}$ admits as local orientation vector:
\begin{equation}\label{eq:orientation_not_prescribed}
\vect{\vec n}(\vect{x_0})=\frac{\vect{u}(\alpha(\vect{x_0}))+\langle \vect{u}(\alpha(\vect{x_0}))^{\perp},\vect{x_0}\rangle \nabla \alpha(\vect{x_0})}{\|\vect{u}(\alpha(\vect{x_0}))+\langle \vect{u}(\alpha(\vect{x_0}))^{\perp},\vect{x_0}\rangle \nabla \alpha(\vect{x_0})\|}\;.
\end{equation}
with $\vect{u}(\alpha(\vect{x_0}))=(\cos (\alpha_0(\vect{x_0})), \sin (\alpha_0(\vect{x_0}))$.
\end{proposition}

\begin{proof}Since the function $\alpha$ is assumed to be $C^1$, the deformation $\ope\Phi$ \eqref{eq:deformation} is also $C^1$. Its Jacobian matrix is given by
\[
\ope{D}\ope{\Phi}(\vect{x})=\begin{pmatrix}\cos \alpha(\vect{x})+\pd{\alpha}{x_1}(\vect{x})\ope{\Phi}_2(\vect{x}) & \sin \alpha(\vect{x})+\frac{\partial \alpha}{\partial x_2}(\vect{x})\ope{\Phi}_2(\vect{x}) \\ -\sin \alpha(\vect{x})-\frac{\partial \alpha}{\partial x_1}(\vect{x})\ope{\Phi}_1(\vect{x}) & \cos \alpha(\vect{x})-\frac{\partial \alpha}{\partial x_2}(\vect{x})\ope{\Phi}_1(\vect{x})\end{pmatrix}\;,
\]
whose determinant is
\[
\mathrm{det}\,\ope{D}\ope{\Phi}(\vect{x})=1+\frac{\partial \alpha}{\partial x_1}(\vect{x})x_2-\frac{\partial \alpha}{\partial x_2}(\vect{x})x_1\;.
\]
Under the assumption (\ref{eq:conditions_alpha}) followed by $\alpha$, the determinant on the open set $U$ is non-zero, so $\ope{\Phi}$ is a $C^1$-diffeomorphism on $U$. Then, Proposition \ref{pro:ex2} and \ref{orientation_deformation} hold, and at each point $\vect{x_0}\in U$, $Z_{\ope{\Phi},X}$ admits as local orientation vector $\vect{\vec n}(\vect{x_0})=\ope{D}\ope{\Phi}(\vect{x_0})^{\T}\vect{e_1}/\|\ope{D}\ope{\Phi}(\vect{x_0})^{\T}\vect{e_1}\|$, which writes
\[
\vect{\vec n}(\vect{x_0})=\frac{\vect{u}(\alpha(\vect{x_0}))+\langle \vect{u}(\alpha(\vect{x_0}))^{\perp},\vect{x_0}\rangle \nabla \alpha(\vect{x_0})}{\|\vect{u}(\alpha(\vect{x_0}))+\langle \vect{u}(\alpha(\vect{x_0}))^{\perp},\vect{x_0}\rangle \nabla \alpha(\vect{x_0})\|}\;.
\]
\end{proof}

\begin{figure}
\centering
\begin{tabular}{cc}
\includegraphics[scale=0.5]{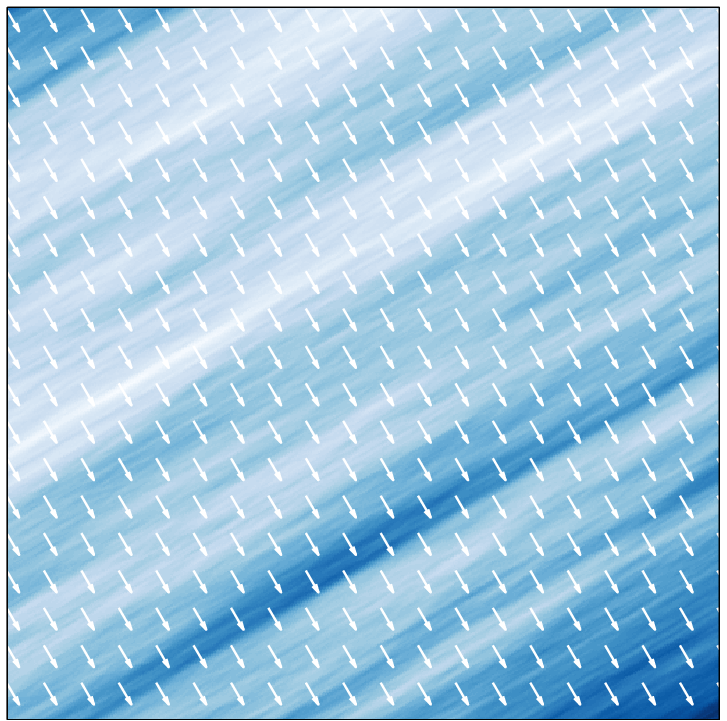} & \includegraphics[scale=0.5]{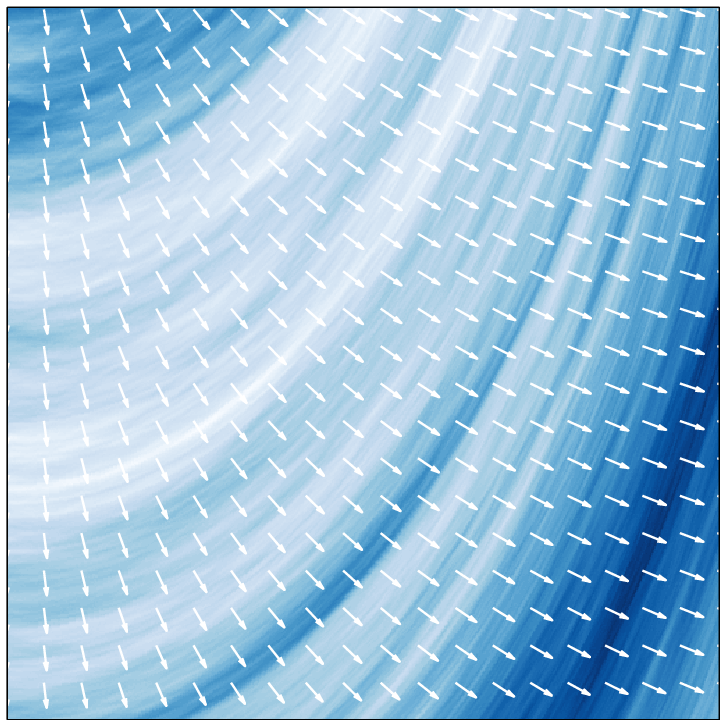}\\
\includegraphics[scale=0.5]{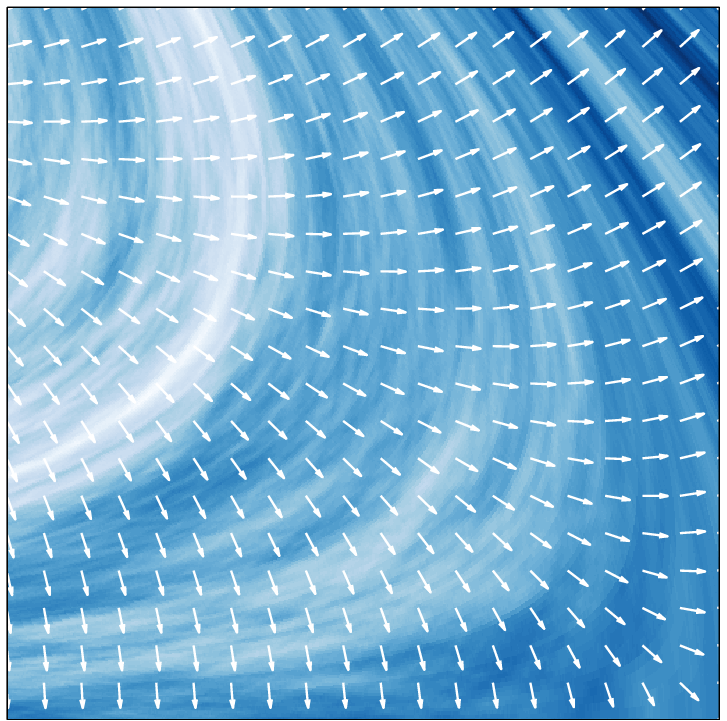} & \includegraphics[scale=0.5]{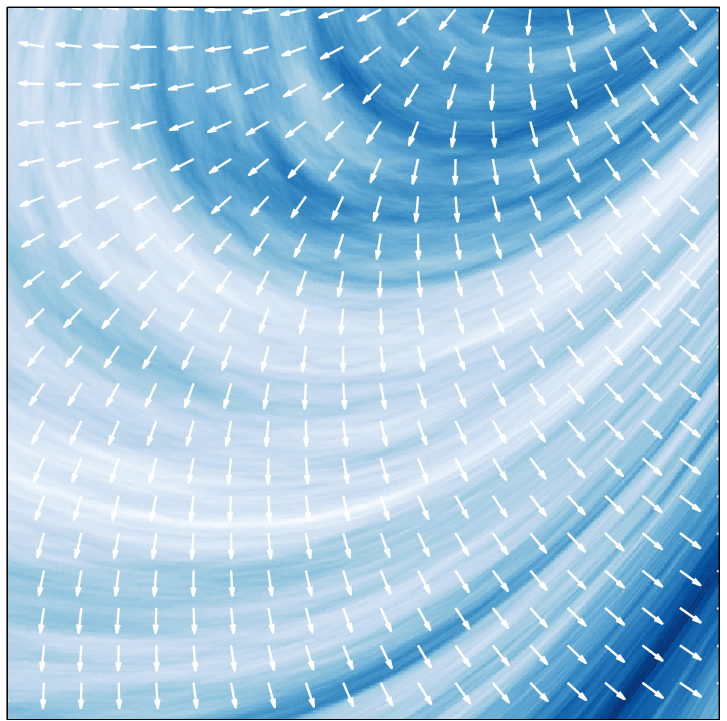}
\end{tabular}
\caption{Image texture of size $512\times 512$ resulting from the simulation of the field $Z_{\ope{\Phi},X}(\vect{x})=X(\mathbf{R}_{-\alpha(\vect{x})}\vect{x})$ on $[0,1]^2$, where $X$ is the standard elementary field with parameters $H=0.5$, $\alpha_0=0$ and $\delta=0.3$, for the following functions $\alpha$: (a) $\alpha(x_1,x_2)=-\frac \pi 3$ (top left), (b) $\alpha(x_1,x_2)=-\frac \pi 2+x_1$ (top right), (c) $\alpha(x_1,x_2)=-\frac \pi 2+x_2$ (bottom left), (d)  $\alpha(x_1,x_2)=-\frac \pi 2+x_1^2-x_2$ (bottom right).}
\label{fig:Z_simu}
\end{figure}
Some examples of realizations of $Z_{\ope{\Phi},X}$ on the domain $[0,1]^2$ are displayed on Figure \ref{fig:Z_simu} for different values of $\alpha$, fulfilling the condition (\ref{eq:conditions_alpha}). Remark that the orientation vector given by (\ref{eq:orientation_not_prescribed}) is equal to $\vect{u}(\alpha(\vect{x_0}))=(\cos \alpha(\vect{x_0}),\sin \alpha(\vect{x_0}))$ plus a term depending on the gradient of $\alpha$. Consequently, we do not exactly have a prescribed orientation governed by $\alpha$. 

\

We now inverse the problem, and investigate the construction of a deformation $\ope{\Phi}$, to obtain a prescribed orientation $\alpha$. To this aim, we will employ a conformal deformation, which has the particularity to preserve the angles. 
%
% In order to address this problem, we will see which assumptions are needed on $\alpha$ to construct a deformation $\ope{\Phi}$ leading to a local orientation at $\vect{x_0}$ which is exactly $\alpha(\vect{x_0})$. 
%First we have this result:
%
%%We then apply the result of Proposition~\ref{pro:ex2} in the special case where $\ope{\Phi}$ is a conformal mapping.
%%
%\begin{proposition}\label{prop:Z_orientation}
%Let $\ope{\Phi}=\begin{pmatrix}\ope{\Phi}_1\\ \ope{\Phi}_2\end{pmatrix}$ be a conformal mapping $\mathbb{R}^2 \to \mathbb{R}^2$. Assume that $\ope{\Phi}$ is continuously differentiable on $\mathbb{R}^2$, that its differential never vanishes and that
%\begin{equation}\label{e:holom}
%\frac{\partial \ope{\Phi}_1}{\partial x_1}=\frac{\partial \ope{\Phi}_2}{\partial x_2}\quad \text{and}\quad \frac{\partial \ope{\Phi}_1}{\partial x_2}=-\frac{\partial \ope{\Phi}_2}{\partial x_1}\;.
%\end{equation}
%Then, at each point $\vect{x_0}$ a local orientation of $Z$ is
%\[
%\vect{\vec n}_Z(\vect{x_0})=\frac{\ope{D}\ope{\Phi}(\vect{x_0})^{\T}\vect{e_1}}{\|\ope{D}\ope{\Phi}(\vect{x_0})^{\T}\vect{e_1}\|}\;.
%\]
%\end{proposition}
%%
%\begin{proof}
%The continuous differentiability of $\ope{\Phi}$ and Condition~(\ref{e:holom}) exactly means that $\ope{\Phi}$ is a conformal mapping from $\mathbb{R}^2$ to $\mathbb{R}^2$, which means that its differential is a direct similitude. Proposition~\ref{pro:ex2} then directly implies the result.
%\end{proof}
%
An important result,  stated  in the following proposition, is that we can prescribed the orientation $\alpha$ of a Gaussian field, if this orientation  is supposed to be harmonic. 
\begin{proposition}\label{cor:harmonic}
Let $Z_{\ope{\Phi},X}(\vect{x})$ be the Gaussian field \eqref{e:model-perrin}, warped by a conformal deformation $\ope{\Phi}$ defined as follows:
let $\alpha: \mathbb{R}^2 \to \R$ be an harmonic function, and $\lambda$ its harmonic conjugate function such that $\ope{\Psi}=\begin{pmatrix}\lambda\\-\alpha\end{pmatrix}$ is holomorphic (as a complex function, identifying $\R^2$ with $\C$). Define now  $\ope{\Phi}$ as any complex primitive of $\exp(\ope{\Psi})$, as an holomorphic function on $\mathbb{C}$. Then at any point $\vect{x_0}$, an approximate (up to $\delta^2$) local orientation of $Z_{\ope{\Phi},X}$ is 
\[
\vect{\vec n}_Z(\vect{x_0})=\begin{pmatrix}\cos(\alpha(\vect{x_0}))\\ \sin(\alpha(\vect{x_0}))\end{pmatrix}\;,
\]
which is exactly the orientation vector defined by the angle function $\alpha$.
\end{proposition}
\begin{proof}
Firstly, the existence of $\lambda$ is the classical result of the existence of an harmonic conjugate of $\alpha$ (see \cite{Stein70}). Then $\ope{\Psi}$ is holomorphic, and  $\exp(\ope{\Psi})$ is holomorphic too (as the composition of holomorphic functions). In addition, since $\ope{\Phi}$ is a complex primitive of $\exp(\ope{\Psi})$ as an holomorphic function on $\mathbb{C}$, we have at any point:
\[
\ope{\Phi}'(\vect{x_0})=\exp(\ope{\Psi}(\vect{x_0}))=e^{\lambda (\vect{x_0})}~e^{-i \alpha(\vect{x_0})}
\]
(as a complex function in $\C$).
Moreover, since $\ope{\Phi}$ is holomorphic, 
\[
\ope{\Phi}'(\vect{x_0}) = \frac{\partial \ope{\Phi}}{\partial x_1}(\vect{x_0})=-i ~\frac{\partial \ope{\Phi}}{\partial x_2}(\vect{x_0})\;,
\]
which leads to the Jacobian matrix:
\[
\ope{D}\ope{\Phi}(\vect{x_0})=\exp(\lambda(\vect{x_0}))\begin{pmatrix}\cos(\alpha(\vect{x_0}))&\sin(\alpha(\vect{x_0}))\\ -\sin(\alpha(\vect{x_0}))&\cos(\alpha(\vect{x_0}))\end{pmatrix}
\]
and concludes the proof.
\end{proof}

%\noindent We now assume that we are given some consider differentiable function $\alpha$. We then consider the case where 
%\[
%\Phi(x)=\mathbf{R}_{-\alpha(x)}\cdot x=\begin{pmatrix}\cos \alpha(x)\, x_1+\sin \alpha(x)\, x_2\\ -\sin \alpha(x)\, x_1+\cos \alpha(x)\, x_2\end{pmatrix}\equiv \begin{pmatrix}\Phi_1(x)\\ \Phi_2(x)\end{pmatrix}
%\]

\begin{example}[Affine orientation functions]
We consider the family of harmonic functions $$\alpha(x_1,x_2)=ax_1+bx_2+c~,$$
 with $a,b,c$ real constants. By the procedure of Proposition~\ref{cor:harmonic}, we are able to construct the deformation function $\ope{\Phi}$, whose explicit formula is
\begin{equation}\label{eq:harmonic_family}
\ope{\Phi}(x_1,x_2)=\frac{\exp(ax_2-bx_1)}{a^2+b^2}\begin{pmatrix}a\sin(ax_1+bx_2+c)-b\cos(ax_1+bx_2+c)\\ a\cos(ax_1+bx_2+c)+b\sin(ax_1+bx_2+c)\end{pmatrix}\;.
\end{equation}
Then we can verify that 
\[
\ope{D}\ope{\Phi}(\vect{x})^{\T}\vect{e_1}=\exp(ax_2-bx_1)\begin{pmatrix}\cos(ax_1+bx_2+c)\\ \sin(ax_1+bx_2+c)\end{pmatrix},\quad \vect{\vec n}(\vect{x})=\begin{pmatrix}\cos \alpha(\vect{x})\\ \sin \alpha(\vect{x})\end{pmatrix}\;.
\]
An example of simulation of such a prescribed local orientation field is provided in Figure \ref{fig:harmonic_simu}, where the angle variations are governed by the function $\alpha(x_1,x_2)=2x_1-x_2$.
\end{example}

\begin{figure}
\centering
\includegraphics[scale=0.5]{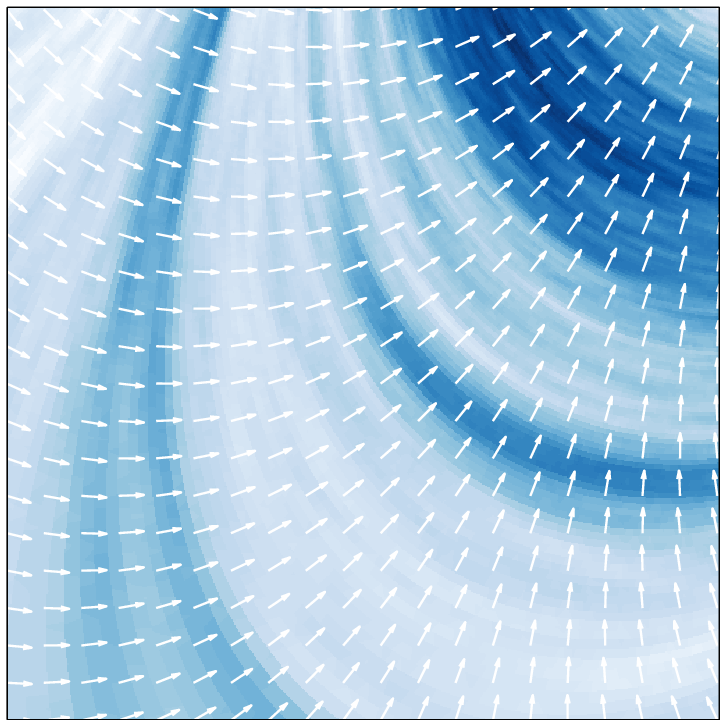} 
\caption{Image texture of size $512\times 512$ resulting from the simulation of the field $Z_{\ope{\Phi},X}(\vect{x})=X(\ope{\Phi}(\vect{x}))$ on $[0,1]^2$, where $X$ is the standard elementary field with parameters $H=0.5$, $\alpha_0=0$ and $\delta=0.3$, and $\ope{\Phi}$ is the deformation function defined by (\ref{eq:harmonic_family}) corresponding to the harmonic function $\alpha(x_1,x_2)=ax_1+bx_2+c$, with $(a,b)=(2,-1)$.}
\label{fig:harmonic_simu}
\end{figure}

%\begin{proposition}\label{prop:rot}
%The Locally Anisotropic Fractional Brownian Field $X$ can be identified in law to
%\begin{equation}
%X(x)\overset{(d)}{=}X_0(\mathbf{R}_{-\alpha(x)}x)\equiv Z(x),
%\end{equation}
%\end{proposition}

%\begin{proof}
%Starting from the definition
%\begin{equation}
%X(x)=\int_{\mathbb{R}^2}(e^{i\langle x,\vect{\xi}\rangle}-1)\frac{\mathds{1}_{[-\delta,\delta]}(\arg \vect{\xi}-\alpha(x))}{\lvert \vect{\xi}\rvert^{H+1}}\diff \hat W(\vect{\xi}),
%\end{equation}
%we do a change of variable $\vect{\xi}=\mathbf{R}_{\alpha(x)}\zeta=\begin{pmatrix}\cos \alpha(x) & -\sin \alpha(x) \\ \sin \alpha(x) & \cos \alpha(x)\end{pmatrix}\zeta$, and $\diff \hat W(\vect{\xi})=\diff \hat W(\mathbf{R}_{\alpha(x)}\zeta)=\diff \hat W(\zeta)$. Since $\arg \vect{\xi}-\alpha(x)=\arg \mathbf{R}_{\alpha(x)}\zeta-\alpha(x)=\arg \zeta$ and $\langle x,\vect{\xi}\rangle =\langle x,\mathbf{R}_{\alpha(x)}\zeta\rangle = \langle \mathbf{R}_{-\alpha(x)}x,\zeta\rangle $ we obtain: 
%\begin{align}
%X(x)
%&=\int_{\mathbb{R}^2}(e^{i\langle \mathbf{R}_{-\alpha(x)}x,\zeta\rangle }-1)\frac{\mathds{1}_{[-\delta,\delta]}(\arg \zeta)}{\lvert \zeta\rvert^{H+1}}\diff \hat W(\zeta)\\ 
%&=X_0(\mathbf{R}_{-\alpha(x)}x)
%\end{align}
% \end{proof}

%Notice that we have exhibited two fields which are equal in law, but have different orientations comparing Proposition \ref{prop:lafbf_orientation} and Proposition \ref{prop:Z_orientation}.

%\appendix
\section{Proofs}\label{s:proofs}
This last section is devoted to the proofs of Theorem~\ref{theo:local:orientation:genCFBA} and Proposition~\ref{pro:ex2}.
\subsection{Technical lemmas}
We first state and prove some lemmas that are used in the proof of Theorem~\ref{theo:local:orientation:genCFBA}.
\begin{lemma}\label{lem:1}
Assume that $h$ is a $\beta$--H\"older function with Lipschitz constant $\Lambda_h$ defined on $\mathbb{R}^2$ satisfying
\[
0<a=\inf_{\vect{x}\in \mathbb{R}^2} h(\vect{x})\leq \sup_{\vect{x}\in \mathbb{R}^2} h(\vect{x})=b<\beta\leq 1 \;.
\] 
Then, for all $\vect{x},\vect{y}\in \mathbb{R}^2$ and for all $\vect{\xi}\in \mathbb{R}^2$,
\[
\left|\lVert \vect{\xi}\rVert^{-h(\vect{y})}-\lVert \vect{\xi}\rVert^{-h(\vect{x})}\right|\leq \Lambda_h\|\vect{y}-\vect{x}\|^\beta\,\abs{\log \lVert \vect{\xi} \rVert}~ \left(\lVert \vect{\xi} \rVert^{-a-1}\mathds{1}_{\lVert \vect{\xi}\rVert>1}+\lVert \vect{\xi} \rVert^{-b-1}\mathds{1}_{\lVert \vect{\xi}\rVert\leq 1}\right)\;.
\] 
\end{lemma}
\begin{proof}
Let us fix $\vect{\xi}\in \mathbb{R}^2$ and apply the mean value inequality to the function  
\[
h\mapsto \lVert \vect{\xi} \rVert^{-h-1}=\exp(-(h+1) \log  \| \vect{\xi} \|)\;.
\]
We obtain that
\[
\left|\lVert \vect{\xi} \rVert^{-h_2-1}-\lVert \vect{\xi} \rVert^{-h_1-1}\right|\leq |h_1-h_2|\,\abs{\log \lVert \vect{\xi} \rVert}~ \left(\lVert \vect{\xi} \rVert^{-\alpha-1}\right)\;,
%&=|\exp(-(h_2+1)\log(\lVert \vect{\xi} \rVert))-\exp(-(h_1+1)\log(\lVert \vect{\xi} \rVert))|\;,\\
\]
with
$\alpha=\min (h_1,h_2)$ if $\lVert \vect{\xi}\rVert>1$ and $\alpha=\max (h_1,h_2)$ if $\lVert \vect{\xi} \rVert<1$.
This leads to the inequality:
\begin{align*}
&\forall (\vect{x},\vect{y})\in (\mathbb{R}^2)^2,\,\forall \vect{\xi}\in \mathbb{R}^2,\quad \left|\lVert \vect{\xi}\rVert^{-h(\vect{y})}-\lVert \vect{\xi}\rVert^{-h(\vect{x})}\right|\leq  \lvert h(\vect{y})-h(\vect{x})\rvert\, \abs{\log \lVert \vect{\xi} \rVert}\\
&\qquad \qquad \qquad \qquad \qquad \qquad\times \left(\lVert \vect{\xi} \rVert^{-\inf_{\vect{z}} h(\vect{z})-1}\mathds{1}_{\lVert \vect{\xi}\rVert>1}+\lVert \vect{\xi} \rVert^{-\sup_{\vect{z}} h(\vect{z})-1}\mathds{1}_{\lVert \vect{\xi}\rVert\leq 1}\right)\;.
\end{align*}
The holderianity of $h$ allows to conclude.
\end{proof}
\begin{lemma}\label{lem:2}
Assume that $h : \mathbb{R}^2\to [0,1]$ and $C : \mathbb{R}^2\times \mathbb{R}^2\to \mathbb{R}_+$ are two functions satisfying assumptions ($\mathcal{H}$). Let $\vect{x}\in \mathbb{R}^2$. Then, there exists some constant $K_{\vect{x}}>0$ depending only on $\vect{x}$ such that for any $(\rho,\vect{w})\in \mathbb{R}_+\times \mathbb{R}^2$ with $|\rho|\leq 1$ and $\norme{\vect{w}}\leq 1$ we have
\[
\int_{\mathbb{R}^2}\left|\expj{\ip{\vect{x}}{\vect{\xi}}}-1\right|^2\left[f^{1/2}(\vect{x}+\rho \vect{w},\vect{\xi})-f^{1/2}(\vect{x},\vect{\xi})\right]^2\dif \vect{\xi} \leq K_{\vect{x}}|\rho|^{2\beta}\max(\|\vect{w}\|^{2\beta},\|\vect{w}\|^{2\eta})\left(1+\|\vect{x}\|^2\right)\;.
\]
Moreover the function $\vect{x}\mapsto K_{\vect{x}}$ is bounded on any compact set.
\end{lemma}
\begin{proof}
Observe that
\begin{equation*}
\frac{C(\vect{x}+\rho \vect{w},\vect{\xi})}{\lVert \vect{\xi} \rVert^{h(\vect{x}+\rho \vect{w})+1}}-\frac{C(\vect{x},\vect{\xi})}{\lVert \vect{\xi} \rVert^{h(\vect{x})+1}}=\frac{C(\vect{x}+\rho \vect{w},\vect{\xi})}{\lVert \vect{\xi} \rVert^{h(\vect{x}+\rho \vect{w})+1}}-\frac{C(\vect{x}+\rho \vect{w},\vect{\xi})}{\lVert \vect{\xi} \rVert^{h(\vect{x})+1}}+\frac{C(\vect{x}+\rho \vect{w},\vect{\xi})}{\lVert \vect{\xi} \rVert^{h(\vect{x})+1}}-\frac{C(\vect{x},\vect{\xi})}{\lVert \vect{\xi} \rVert^{h(\vect{x})+1}}\;.
\end{equation*}
Using the classical inequality $|a-b|^2\leq 2(|a|^2+|b|^2$), we deduce that
\begin{eqnarray*}
&&\int_{\mathbb{R}^2}\left|\expj{\ip{\vect{x}}{\vect{\xi}}}-1\right|^2\left[f^{1/2}(\vect{x}+\rho \vect{w},\vect{\xi})-f^{1/2}(\vect{x},\vect{\xi})\right]^2\dif \vect{\xi}\;,\\
&\leq& 2\int_{\mathbb{R}^2}\left|\expj{\ip{\vect{x}}{\vect{\xi}}}-1\right|^2\left[\frac{C(\vect{x}+\rho \vect{w},\vect{\xi})}{\lVert \vect{\xi} \rVert^{h(\vect{x}+\rho \vect{w})+1}}-\frac{C(\vect{x}+\rho \vect{w},\vect{\xi})}{\lVert \vect{\xi} \rVert^{h(\vect{x})+1}}\right]^2\dif\vect{\xi}\\
&&+2\int_{\mathbb{R}^2}\frac{\left|\expj{\ip{\vect{x}}{\vect{\xi}}}-1\right|^2}{\lVert \vect{\xi} \rVert^{2h(\vect{x})+2}}\left[C(\vect{x}+\rho \vect{w},\vect{\xi})-C(\vect{x},\vect{\xi})\right]^2\dif\vect{\xi}\;.
\end{eqnarray*}

To bound the latter integral $\int_{\mathbb{R}^2}\frac{\left|\expj{\ip{\vect{x}}{\vect{\xi}}}-1\right|^2}{\lVert \vect{\xi} \rVert^{2h(\vect{x})+2}}\left[C(\vect{x}+\rho \vect{w},\vect{\xi})-C(\vect{x},\vect{\xi})\right]^2\dif\vect{\xi}$, we set $\vect{\xi}=r\vect{\Theta}$ with $(r,\vect{\Theta})\in \mathbb{R}_+^*\times \mathbb{S}^1$ and use the homogeneity of $C$. It yieds:
\begin{eqnarray*}
&&\int_{\mathbb{R}^2}\frac{\left|\expj{\ip{\vect{x}}{\vect{\xi}}}-1\right|^2}{\lVert \vect{\xi} \rVert^{2h(\vect{x})+2}}\left[C(\vect{x}+\rho \vect{w},\vect{\xi})-C(\vect{x},\vect{\xi})\right]^2\dif\vect{\xi}\;,\\
%&=&\int_{\mathbb{R}_+^*\times \mathbb{S}^1}\left|\expj{\ip{\vect{x}}{r\vect{\Theta}}}-1\right|^2\left[\frac{C(\vect{x}+\rho \vect{w},\vect{\Theta})}{r^{h(\vect{x})+1}}-\frac{C(\vect{x},\vect{\Theta})}{r^{h(\vect{x})+1}}\right]^2 r\dif r\dif\vect{\Theta}\;\\
&=&\int_{\mathbb{R}_+^*}\frac{|\expj{s}-1|^2}{{s^{2h(\vect{x})+1}}}\left[\int_{\mathbb{S}^1}|\ip{\vect{x}}{\vect{\Theta}}|^{2h(\vect{x})}\left[C(\vect{x}+\rho \vect{w},\vect{\Theta})-C(\vect{x},\vect{\Theta})\right]^2 \dif\vect{\Theta}\right]\dif s\;,\\
&\leq& \|\vect{x}\|^{2h(\vect{x})}\left[\int_{\mathbb{R}_+^*}\frac{|\expj{s}-1|^2}{{s^{2h(\vect{x})+1}}}\dif s\right]\left[\int_{\mathbb{S}^1}\left[C(\vect{x}+\rho \vect{w},\vect{\Theta})-C(\vect{x},\vect{\Theta})\right]^2 \dif\vect{\Theta}\right]\;,
\end{eqnarray*}
where we set $s=r \ip{\vect{x}}{\vect{\Theta}}$ in the second equality.
We now use condition \eqref{eq:homogene} of assumptions~($\mathcal{H}$) with $\vect{z}=\rho \vect{w}\in B(\vect{0},1)$. Hence,
\[
\int_{\mathbb{R}^2}\frac{\left|\expj{\ip{\vect{x}}{\vect{\xi}}}-1\right|^2}{\lVert \vect{\xi} \rVert^{2h(\vect{x})+2}}\left[C(\vect{x}+\rho \vect{w},\vect{\xi})-C(\vect{x},\vect{\xi})\right]^2\dif\vect{\xi}\leq A_{\vect{x}}~ |\rho|^{2\eta}\|\vect{w}\|^{2\eta}\|\vect{x}\|^{2h(\vect{x})}\left[\int_{\mathbb{R}_+^*}\frac{|\expj{s}-1|^2}{{s^{2h(\vect{x})+1}}}\dif s\right].
\]
Then, since $\|\vect{x}\|^{2h(\vect{x})}\leq \|\vect{x}\|^{2}+1$ is always valid ($h(\mathbb{R}^2)\subset [0,1]$), one has
\begin{equation}
\label{eq:CV-J}
\int_{\mathbb{R}^2}\frac{\left|\expj{\ip{\vect{x}}{\vect{\xi}}}-1\right|^2}{\lVert \vect{\xi} \rVert^{2h(\vect{x})+2}}\left[C(\vect{x}+\rho \vect{w},\vect{\xi})-C(\vect{x},\vect{\xi})\right]^2\dif\vect{\xi}\leq  B_{\vect{x}}\|\vect{w}\|^{2\eta} |\rho|^{2\eta}\left(\|\vect{x}\|^2+1\right)\;,
\end{equation}
with
\[
B_{\vect{x}}=A_{\vect{x}}\int_{\mathbb{R}_+^*}\frac{|\expj{s}-1|^2}{{s^{2h(\vect{x})+1}}}\dif s<\infty\;.
\]
We now bound $\int_{\mathbb{R}^2}\left|\expj{\ip{\vect{x}}{\vect{\xi}}}-1\right|^2\left[\frac{C(\vect{x}+\rho \vect{w},\vect{\xi})}{\lVert \vect{\xi} \rVert^{h(\vect{x}+\rho \vect{w})+1}}-\frac{C(\vect{x}+\rho \vect{w},\vect{\xi})}{\lVert \vect{\xi} \rVert^{h(\vect{x})+1}}\right]^2\dif\vect{\xi}$. Since $C$ is bounded and by Lemma~\ref{lem:1} we have for some $A>0$ depending only on $h$ and $C$:
\begin{eqnarray*}
&&\int_{\mathbb{R}^2}\left|\expj{\ip{\vect{x}}{\vect{\xi}}}-1\right|^2\left[\frac{C(\vect{x}+\rho \vect{w},\vect{\xi})}{\lVert \vect{\xi} \rVert^{h(\vect{x}+\rho \vect{w})+1}}-\frac{C(\vect{x}+\rho \vect{w},\vect{\xi})}{\lVert \vect{\xi} \rVert^{h(\vect{x})+1}}\right]^2\dif\vect{\xi}\\
&\leq & \Lambda_{h,C} |\rho|^{2\beta}\|\vect{w}\|^{2\beta} \int_{\mathbb{R}^2}\left|\expj{\ip{\vect{x}}{\vect{\xi}}}-1\right|^2 \left[\left|\log\|\vect{\xi}\|\right|^2\left(\|\vect{\xi}\|^{-2a-2}\mathds{1}_{\|\vect{\xi}\|>1}+\|\vect{\xi}\|^{-2b-2}\mathds{1}_{\|\vect{\xi}\|\leq 1}\right)\right]\dif\vect{\xi}\;.
\end{eqnarray*}

Since $\left|\expj{\ip{\vect{x}}{\vect{\xi}}}-1\right|\leq \min (\|\vect{x}\|\cdot\|\vect{\xi}\|,2)$, we directly get that
\begin{equation}
\label{eq:CV-I}
\int_{\mathbb{R}^2}\left|\expj{\ip{\vect{x}}{\vect{\xi}}}-1\right|^2\left[\frac{C(\vect{x}+\rho \vect{w},\vect{\xi})}{\lVert \vect{\xi} \rVert^{h(\vect{x}+\rho \vect{w})+1}}-\frac{C(\vect{x}+\rho \vect{w},\vect{\xi})}{\lVert \vect{\xi} \rVert^{h(\vect{x})+1}}\right]^2\dif\vect{\xi}\leq \widetilde{\Lambda}~ |\rho|^{2\beta}\|\vect{w}\|^{2\beta}\left(\|\vect{x}\|^2+1\right)\;.
\end{equation}
with
\[
\widetilde{\Lambda}=\Lambda_{h,C}\left[\int_{\mathbb{R}^2} \left|\log\|\vect{\xi}\|\right|^2\left(\|\vect{\xi}\|^{-2a-2}\mathds{1}_{\|\vect{\xi}\|>1}+\|\vect{\xi}\|^{-2b}\mathds{1}_{\|\vect{\xi}\|\leq 1}\right)\dif\vect{\xi}\right]\;.
\]
The conclusion then follows from~\eqref{eq:CV-J} and \eqref{eq:CV-I} with $K_{\vect{x}}=2B_{\vect{x}}+2\widetilde{A}$. The fact that $\vect{x}\mapsto K_{\vect{x}}$ is bounded on any compact set comes from the fact that $\vect{x}\mapsto A_{\vect{x}}$ is bounded on any compact set.
\end{proof}
%%%%%%%%%%%%%%%%%%%%%%%%%%%%%%%
\subsection{Proof of Theorem~\ref{theo:local:orientation:genCFBA}\\}\label{s:proof:local:orientation}
\label{proof:prop9}

% Y_x0, T_x0
% mbm mbf
% ponctuations

Let $X$ be the Gaussian field defined by formula (\ref{eq:gafbf}), and $\vect{x_0}\in \mathbb{R}^2$. Let $Z_{\vect{x_0}}$ be the Gaussian field 
\[
Z_{\vect{x_0},\rho}(\vect{u})=\frac{X(\vect{x_0}+\rho \vect{u})-X(\vect{x_0})}{\rho^{h(\vect{x_0)}}}\;,
\]
and $Y_{\vect{x_0}}$ the $H$-sssi field defined by formula (\ref{eq:Yx0}). We are going to prove that $Y_{\vect{x_0}}$ is the tangent field of $X$ at $\vect{x_0}\in \mathbb{R}^2$, that is
\[
\left\{Z_{\vect{x_0},\rho}(\vect{h})\right\}_{\vect{h}\in \mathbb{R}^2}\overset{d}{\longrightarrow} \left\{Y_{\vect{x_0}}(\vect{h})\right\}_{\vect{h}\in \mathbb{R}^2}\;.
\]
as $\rho\to 0$, in the sense of weak convergence of stochastic processes. The proof is divided in two steps~:
\begin{itemize}
\item[(i)] We first prove that the finite dimensional distribution of $Z_{\vect{x_0},\rho}$ converge to those of $Y_{\vect{x_0}}$ as $\rho\to 0$:
\[
(Z_{\vect{x_0},\rho}(\vect{h}_1),\dotsc, Z_{\vect{x_0},\rho}(\vect{h}_N))\longrightarrow (Y_{\vect{x_0}}(\vect{h}_1),\dotsc,Y_{\vect{x_0}}(\vect{h}_N))\;,
\]
which means the convergence of the measures of these finite dimensional random vectors on $ \mathbb{R}^N$. The L\'evy theorem insures that it is equivalent to prove the converge in term of the characteristic functions of these random vectors, which is, in the Gaussian case, equivalent to show that we have convergence with respect to the covariance:
\begin{equation}\label{eq:convergence_cov}
\forall (\vect{u},\vect{v})\in (\mathbb{R}^2)^2,\quad \lim_{\rho\to 0}\mathbb{E}[Z_{\vect{x_0},\rho}(\vect{u})\overline{Z_{\vect{x_0},\rho}(\vect{v})}]=\mathbb{E}(Y_{\vect{x_0}}(\vect{u})\overline{Y_{\vect{x_0}}(\vect{v})})\;.
\end{equation}
\item[(ii)] Thereafter, we set $\rho_n=1/n\in [0,1]$ and prove that the sequence of random fields $(Z_n)_{n\in \mathbb{N}^{\ast}}\overset{\mathrm{def}}{=}(Z_{\vect{x_0},\rho_n})_{n\in \mathbb{N}^{\ast}}$, satisfies a tightness property, which is fulfilled if $(Z_n)_{n\in \mathbb{N}^{\ast}}$ satisfies the following Kolmogorov criteria (see for example \cite{KarShr88} p.64):
\begin{equation}\label{eq:kolmo_criteria}
\forall T>0, \, \forall \vect{u}, \vect{v}\in [-T,T]^{2},\quad \sup\limits_{n>1}\mathbb{E}(|Z_n(\vect{u})-Z_n(\vect{v})|^{\gamma_1})\leq C_{0}(T)\lVert \vect{u}-\vect{v}\rVert^{2+\gamma_2}\;,
\end{equation}
for some positive constant $C_{0}(T)$ which may depend on $T$ and $\gamma_1,\gamma_2$ which are universal positive constants.
\end{itemize}

\begin{remark}\label{rem:gaussian_moment}
Since $Z_{\vect{x_0},\rho}(\vect{u})-Z_{\vect{x_0},\rho}(\vect{v})$ is a Gaussian vector, then for all $\gamma_1>0$
\[
\sup_{\rho\in (0,1)}\|\vect{u}-\vect{v}\|^{-\gamma_0\gamma_1}\mathbb{E}\left[\left|Z_{\vect{x_0},\rho}(\vect{u})-Z_{\vect{x_0},\rho}(\vect{v})\right|^{\gamma_1}\right]
\]
and
\[
\left[\sup_{\rho\in (0,1)}\|\vect{u}-\vect{v}\|^{-2\gamma_0}\mathbb{E}\left|Z_{\vect{x_0},\rho}(\vect{u})-Z_{\vect{x_0},\rho}(\vect{v})\right|^{2}\right]^{\gamma_1/2}
\]
are equal up to a multiplicative constant depending only on $\gamma_1$. The inequality~\eqref{eq:kolmo_criteria} is satisfied by considering $\gamma_1>2/\gamma_0$ with $\gamma_2=\gamma_0\gamma_1-2$. Therefore, it will be sufficient to verify that $0\leq \gamma_0\leq 1$ such that
\begin{equation}\label{eq:kolmo_criteria0}
\forall T>0, \, \forall \vect{u}, \vect{v}\in [-T,T]^{2},\quad \sup\limits_{n>1}\mathbb{E}(|Z_n(\vect{u})-Z_n(\vect{v})|^{2})\leq C_{0}(T)\lVert \vect{u}-\vect{v}\rVert^{2\gamma_0}\;.
\end{equation}
\end{remark}

\begin{remark}\label{rem:kolmo}
Since the notion of tangent field is a local notion at point $\vect{x}_0$, it is equivalent to determine the tangent field of $X$ at $\vect{x}_0$ or that of $\widetilde X(\vect{x})=\varphi_{\vect{x}_0}(\vect{x}) X(\vect{x})$ at point $\vect{x}_0$, where $\varphi$ is a $C^{\infty}$ function which is equal to 1 on a neighborhood of $\vect{x}_0$ (e.g. the ball $B(\vect{x}_0,a)$ of radius $a$) then decreases and vanished outside a compact set (e.g. the ball $B(\vect{x}_0,b)$). Such bump functions can easily be constructed \cite{tu2008bump}, from a 1-D function $\varphi_0$ as illustrated in Figure \ref{fig:bump} by taking $\varphi_{\vect{x}_0}(\vect{x})=\varphi_0( \norme{\vect{x}-\vect{x}_0})$. Then, we define
\[
\widetilde Z_{\vect{x_0},\rho}(\vect{u})=\frac{\widetilde X(\vect{x_0}+\rho \vect{u})-\widetilde X(\vect{x_0})}{\rho^{h(\vect{x_0)}}}\;.
\]
In step (i), computing the limit for given $\vect{u}$ and $\vect{v}$, nothing changes since from a certain rank, $\rho_n$ is such that $\vect{x_0}+\rho_n \vect{u}$ and $\vect{x_0}+\rho_n \vect{v}$ fall in the ball $B(\vect{x}_0,a)$ in which $\widetilde X=X$.\\
In step (ii), with $\widetilde Z_n(\vect{u})-\widetilde Z_n(\vect{v})=\rho_n^{-h(\vect{x}_0)}(\widetilde X(\vect{x_0}+\rho_n \vect{u})-\widetilde X(\vect{x_0}+\rho_n \vect{v}))$, the following inequality
\begin{equation}\label{eq:kolmo_criteria_bis}
\sup\limits_{n>1}\mathbb{E}(|\widetilde Z_n(\vect{u})-\widetilde Z_n(\vect{v})|^2)\leq C_{0}(T)\lVert \vect{u}-\vect{v}\rVert^{2\gamma_0}\;,
\end{equation}
requires to distinguish three cases:\\
\begin{itemize}
\item If $\vect{u}$ and $\vect{v}$ are in $B(\vect{x}_0,b)^c$, then $\widetilde X(\vect{x_0}+\rho_n \vect{u})=\widetilde X(\vect{x_0}+\rho_n \vect{v})=0$ and so \eqref{eq:kolmo_criteria_bis} is satisfied.
\item If $\vect{u}\in B(\vect{x}_0,b)$ and $\vect{v}\in B(\vect{x}_0,b)^c$, then 
\begin{align*}
\widetilde Z_n(\vect{u})-\widetilde Z_n(\vect{v})
&=\rho_n^{-h(\vect{x}_0)}\widetilde X(\vect{x_0}+\rho_n \vect{u})\;,\\
&=\rho_n^{-h(\vect{x}_0)}\left(\varphi_{\vect{x}_0}(\vect{x_0}+\rho_n \vect{u})-\varphi_{\vect{x}_0}(\vect{x_0}+\rho_n \vect{v})\right)X(\vect{x_0}+\rho_n \vect{u})\;,
\end{align*}
with $\abs{\varphi_{\vect{x}_0}(\vect{x_0}+\rho_n \vect{u})-\varphi_{\vect{x}_0}(\vect{x_0}+\rho_n \vect{v})}^2\leq A(T)\rho_n^2 \norme{\vect{u}-\vect{v}}^2$ since $\varphi$ is $C^{\infty}$ that is a Lipschitz function and $ \mathbb{E}\left[X(\vect{x_0}+\rho_n \vect{u})^2\right]\leq B(T)$ by continuity of the covariance function, on the compact set $[-T,T]^2$. Again, \eqref{eq:kolmo_criteria_bis}  is satisfied.
\item It remains to deal with the case where $\vect{u}$ and $\vect{v}$ are in $B(\vect{x}_0,b)$, that is to say we can restrict ourselves for the proof to the case where $\vect{u}$ and $\vect{v}$ are in the neighborhood of $\vect{x}_0$ as small as we want. We will take for the purposes of the demonstration $b=1/2$, in other words we will be able to restrict ourselves to the compact set $[-T,T]=[-1/2,1/2]$.
\end{itemize}

\end{remark}

\begin{figure}
\centering
\includegraphics[scale=0.2]{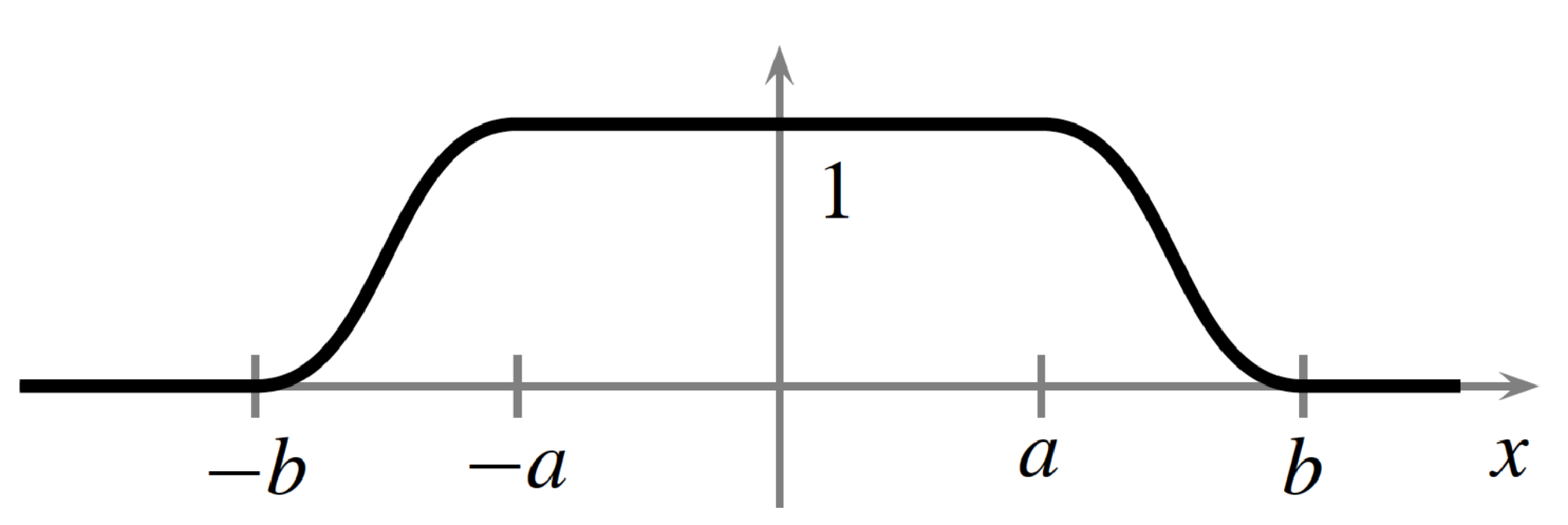}
\caption{Ilustration of a bump function $\varphi_0$.}
\label{fig:bump}
\end{figure}

Now we prove these two conditions (\ref{eq:convergence_cov}) and (\ref{eq:kolmo_criteria}).\\

\noindent (i) \underline{First step}:  \\

We aim at proving (\ref{eq:convergence_cov}) that is for all $(\vect{u},\vect{v})\in (\mathbb{R}^2)^2$:
\begin{equation}\label{e:tangentfield}
\lim_{\rho\to 0}\frac{\mathbb{E}[(X(\vect{x_0}+\rho \vect{u})-X(\vect{x_0}))(\overline{X(\vect{x_0}+\rho \vect{v})-X(\vect{x_0})})]}{\rho^{2h(\vect{x_0})}}=\mathbb{E}(Y_{\vect{x_0}}(\vect{u})\overline{Y_{\vect{x_0}}(\vect{v})})\;.
\end{equation}

Let us take $(\vect{u},\vect{v})\in (\mathbb{R}^2)^2$ and consider $\rho\leq \rho_0$ such that $\rho_0\vect{u},\rho_0\vect{v}\in B(\vect{0},1/2)$.\\

First observe that, by definition of $X$, one has
\begin{equation}
\mathbb{E}[(X(\vect{x_0}+\rho \vect{u})-X(\vect{x_0}))\overline{(X(\vect{x_0}+\rho \vect{v})-X(\vect{x_0})})]=\int_{\mathbb{R}^2}g_{\vect{u},\vect{v}}(\vect{x_0},\vect{\xi})\dif \vect{\xi}\;,
\end{equation}
where we set
\begin{align*}
g_{\vect{u},\vect{v}}(\vect{x_0},\vect{\xi})=&\left[\left(\expj{\ip{\vect{x_0}+\rho \vect{u}}{\vect{\xi}}}-1\right)f^{1/2}(\vect{x_0}+\rho \vect{u},\vect{\xi})-\left(\expj{\ip{\vect{x_0}}{\vect{\xi}}}-1\right)f^{1/2}(\vect{x_0},\vect{\xi})\right]\\
&\times\left[\left(\eu^{-\jj \ip{\vect{x_0}+\rho \vect{v}}{\vect{\xi}}}-1\right)f^{1/2}(\vect{x_0}+\rho \vect{v},\vect{\xi})-\left(\eu^{-\jj \ip{\vect{x_0}}{\vect{\xi}}}-1\right)f^{1/2}(\vect{x_0},\vect{\xi})\right]\;,
\end{align*}
and $f^{1/2}(\vect{x},\vect{\xi})=C(\vect{x},\vect{\xi})/\|\vect{\xi}\|^{h(\vect{x})+1}$.
We now split the integral into four terms:
\begin{align*}
&\mathbb{E}[(X(\vect{x_0}+\rho \vect{u})-X(\vect{x_0}))\overline{(X(\vect{x_0}+\rho \vect{v})-X(\vect{x_0})})]\\
&=\int_{\mathbb{R}^2}\left(\expj{\rho \ip{\vect{u}}{\vect{\xi}}}-1\right)\left(\eu^{-\jj \rho \ip{\vect{v}}{\vect{\xi}}}-1\right)f^{1/2}(\vect{x_0}+\rho \vect{u},\vect{\xi})f^{1/2}(\vect{x_0}+\rho \vect{v,}\vect{\xi})\dif \vect{\xi}&(\mat{I_1})\\
&+\int_{\mathbb{R}^2}\left(\expj{\rho \ip{\vect{u}}{\vect{\xi}}}-1\right)\left(1-\expj{\ip{\vect{x_0}}{\vect{\xi}}}\right)f^{1/2}(\vect{x_0}+\rho \vect{u},\vect{\xi})\left[f^{1/2}(\vect{x_0}+\rho \vect{v},\vect{\xi})-f^{1/2}(\vect{x_0},\vect{\xi})\right]\dif \vect{\xi}&(\mat{I_2})\\
&+\int_{\mathbb{R}^2}\left(1-\eu^{-\jj \ip{\vect{x_0}}{\vect{\xi}}}\right)\left(\eu^{-\jj \rho \ip{\vect{v}}{\vect{\xi}}}-1\right)\left[f^{1/2}(\vect{x_0}+\rho \vect{u},\vect{\xi})-f^{1/2}(\vect{x_0},\vect{\xi})\right]f^{1/2}(\vect{x_0}+\rho \vect{v},\vect{\xi})\dif \vect{\xi}&(\mat{I_3})\\
&+\int_{\mathbb{R}^2}\left|\expj{\ip{\vect{x_0}}{\vect{\xi}}}-1\right|^2\left[f^{1/2}(\vect{x_0}+\rho \vect{u},\vect{\xi})-f^{1/2}(\vect{x_0},\vect{\xi})\right]\left[f^{1/2}(\vect{x_0}+\rho \vect{v},\vect{\xi})-f^{1/2}(\vect{x_0},\vect{\xi})\right]\dif \vect{\xi} &(\mat{I_4})\\
&=\mat{I_1}+\mat{I_2}+\mat{I_3}+\mat{I_4}~.
\end{align*}
In order to prove \eqref{e:tangentfield}, we now  investigate the behavior of each integral $\mat{I_1},\mat{I_2},\mat{I_3},\mat{I_4}$ when $\rho \to 0$.\\

\noindent $\bullet$ \textbf{Study of the first term $\mat{I_1}$}\\

\noindent
We suppose below that $\rho >0$: indeed, since $\vect{\xi}\mapsto C(\vect{x},\vect{\xi})$ is even, the case $\rho<0$ derives in the same way.  
 In the integral $\mat{I_1}$, we set $\vect{\zeta}=\rho\vect{\xi}$ ($\vect{\zeta}=-\rho\vect{\xi}$ if $\rho <0$), $\dif \vect{\zeta}=\rho^2\dif \vect{\xi}$ and use the explicit expression of $f$, then:
\begin{align*}
&\mat{I_1}=\rho^{h(\vect{x_0}+\rho \vect{u})+h(\vect{x_0}+\rho \vect{v})}\int_{\mathbb{R}^2}\left(\expj{\ip{\vect{u}}{\vect{\zeta}}}-1\right)\left(\eu^{-\jj \ip{\vect{v}}{\vect{\zeta}}}-1\right)\frac{C(\vect{x_0}+\rho \vect{u},\vect{\zeta}/\rho)C(\vect{x_0}+\rho \vect{v},\vect{\zeta}/\rho)}{\lVert \vect{\zeta} \rVert^{h(\vect{x_0}+\rho \vect{u})+h(\vect{x_0}+\rho \vect{v})+2}}\dif \vect{\zeta}\;.
\end{align*}
By homogeneity of $\vect{\xi}\mapsto C(\vect{x},\vect{\xi})$, we deduce:
\begin{align*}
&\rho^{-2h(\vect{x_0})}\mat{I_1}=\rho^{h(\vect{x_0}+\rho \vect{u})+h(\vect{x_0}+\rho \vect{v})-2h(\vect{x_0})}\\
&\times \int_{\mathbb{R}^2}\left(\expj{\ip{\vect{u}}{\vect{\zeta}}}-1\right)\left(\eu^{-\jj \ip{\vect{v}}{\vect{\zeta}}}-1\right)\frac{C(\vect{x_0}+\rho \vect{u},\vect{\zeta})C(\vect{x_0}+\rho \vect{v},\vect{\zeta})}{\lVert \vect{\zeta} \rVert^{h(\vect{x_0}+\rho \vect{u})+h(\vect{x_0}+\rho \vect{v})+2}}\dif \vect{\zeta}\;.
\end{align*}
Observe now that
\begin{equation*}
\rho^{h(\vect{x_0}+\rho \vect{u})+h(\vect{x_0}+\rho \vect{v})-2h(\vect{x_0})}=\exp\left({\log\rho~ [h(\vect{x_0}+\rho \vect{u})+h(\vect{x_0}+\rho \vect{v})-2h(\vect{x_0})]}\right)\;.
\end{equation*}
Using that $h$ is $\beta-$H\"{o}lder, we obtain
\begin{equation*}
|h(\vect{x_0}+\rho \vect{u})+h(\vect{x_0}+\rho \vect{v})-2h(\vect{x_0})|\leq |h(\vect{x_0}+\rho \vect{u})-h(\vect{x_0})|+|h(\vect{x_0}+\rho \vect{v})-h(\vect{x_0})|\lesssim (\lVert \vect{u}\rVert^\beta+\lVert\vect{v}\rVert^\beta)|\rho|^\beta\;.
\end{equation*}
Since $\beta>0$ by assumption, and $\displaystyle \lim_{\rho\to 0^+}\rho^\beta \log\rho=0$, we then deduce the limit:
\begin{equation}\label{eq:lim_rho}
\lim_{\rho\to 0^+}\log\rho~[h(\vect{x_0}+\rho \vect{u})+h(\vect{x_0}+\rho \vect{v})-2h(\vect{x_0})]=0\;,
\end{equation}
and hence $\displaystyle \lim_{\rho\to 0^+}\rho^{h(\vect{x_0}+\rho \vect{u})+h(\vect{x_0}+\rho \vect{v})-2h(\vect{x_0})}=1$.

\medskip

It implies that
\begin{align*}
&\lim_{\rho\to 0^+}\rho^{-2h(\vect{x_0})}\mat{I_1}=\lim_{\rho\to 0^+}\int_{\lVert \vect{\zeta} \rVert\leq 1}\left(\expj{\ip{\vect{u}}{\vect{\zeta}}}-1\right)\left(\eu^{-\jj \ip{\vect{v}}{\vect{\zeta}}}-1\right)\frac{C(\vect{x_0}+\rho \vect{u},\vect{\zeta})C(\vect{x_0}+\rho \vect{v},\vect{\zeta})}{\lVert \vect{\zeta} \rVert^{h(\vect{x_0}+\rho \vect{u})+h(\vect{x_0}+\rho \vect{v})+2}}\dif \vect{\zeta}\\
&+\lim_{\rho\to 0^+}\int_{\lVert \vect{\zeta} \rVert\geq 1}\left(\expj{\ip{\vect{u}}{\vect{\zeta}}}-1\right)\left(\eu^{-\jj \ip{\vect{v}}{\vect{\zeta}}}-1\right)\frac{C(\vect{x_0}+\rho \vect{u},\vect{\zeta})C(\vect{x_0}+\rho \vect{v},\vect{\zeta})}{\lVert \vect{\zeta} \rVert^{h(\vect{x_0}+\rho \vect{u})+h(\vect{x_0}+\rho \vect{v})+2}}\dif \vect{\zeta}\;.
\end{align*}
We now apply the Lebesgue's Dominated Convergence Theorem to each integral separately. We first bound the two integrands as follows:
\begin{align*}
& \forall \lVert \vect{\zeta} \rVert\leq 1,\quad  \left|(\expj{\ip{\vect{u}}{\vect{\zeta}}}-1)(\eu^{-\jj \ip{\vect{v}}{\vect{\zeta}}}-1)\frac{C(\vect{x_0}+\rho \vect{u},\vect{\zeta})C(\vect{x_0}+\rho \vect{v},\vect{\zeta})}{\lVert \vect{\zeta} \rVert^{h(\vect{x_0}+\rho \vect{u})+h(\vect{x_0}+\rho \vect{v})+2}}\right|\leq \frac{M^2\|\vect{u}\|\|\vect{v}\|\lVert \vect{\zeta} \rVert^2}{\lVert \vect{\zeta} \rVert^{2(b+1)}}\;,\\ 
&\forall \lVert \vect{\zeta} \rVert\geq 1, \quad \left|(\expj{\ip{\vect{u}}{\vect{\zeta}}}-1)(\eu^{-\jj \ip{\vect{v}}{\vect{\zeta}}}-1)\frac{C(\vect{x_0}+\rho \vect{u},\vect{\zeta})C(\vect{x_0}+\rho \vect{v},\vect{\zeta})}{\lVert \vect{\zeta} \rVert^{h(\vect{x_0}+\rho \vect{u})+h(\vect{x_0}+\rho \vect{v})+2}}\right|\leq 4M^2\frac{1}{\lVert \vect{\zeta} \rVert^{2(a+1)}}\;,
\end{align*}
In the first line we used $|\expj{t}-1|\leq 2~|\sin(t/2)|\leq |t|$, and the Cauchy-Schwarz inequality applied to the $\R^2$-scalar product. Secondly, since $a>0$ and $b<1$ by assumption, we easily check that the functions $\vect{\zeta}\mapsto \lVert \vect{\zeta} \rVert^{-2b}$ and $\vect{\zeta}\mapsto \lVert \vect{\zeta} \rVert^{-2(a+1)}$ are respectively integrable on $\lVert \vect{\zeta} \rVert\leq 1$ and $\lVert \vect{\zeta} \rVert\geq 1$. The Lebesgue's Dominated Convergence Theorem then implies that
\begin{align*}
&\lim_{\rho\to 0^+}\rho^{-2h(\vect{x_0})}\mat{I_1}=\int_{\mathbb{R}^2}\left(\expj{\ip{\vect{u}}{\vect{\zeta}}}-1\right)\left(\eu^{-\jj \ip{\vect{v}}{\vect{\zeta}}}-1\right)\frac{C(\vect{x_0},\vect{\zeta})^2}{\lVert \vect{\zeta} \rVert^{2h(\vect{x_0})+2}}\dif \vect{\zeta}=\mathbb{E}(Y_{\vect{x_0}}(\vect{u})Y_{\vect{x_0}}(\vect{v}))\;.
\end{align*}
since the functions $h$ and $\vect{x} \to C(\vect{x} , \vect{\zeta})$ are continuous.\\

\noindent $\bullet$ \textbf{Study of the other terms $\mat{I_2},\mat{I_3},\mat{I_4}$}\\

We now prove that the three other integrals $\mat{I_2},\mat{I_3},\mat{I_4}$ are negligible with respect to the first one when $\rho$ is small.\\

\noindent We only detail the negligibility of $\mat{I_2}$, the other cases $\mat{I_3}$ and $\mat{I_4}$ being similar. Using the Cauchy--Schwarz inequality, we get that
\begin{align*}
&\mat{I_2}\leq \left[\int_{\mathbb{R}^2}\left|\expj{\rho \ip{\vect{u}}{\vect{\xi}}}-1\right|^2 f(\vect{x_0}+\rho \vect{u},\vect{\xi})\dif\vect{\xi}\right]^{\sfrac{1}{2}}\\
&\qquad \times \left[\int_{\mathbb{R}^2}\left|\expj{\ip{\vect{x_0}}{\vect{\xi}}}-1\right|^2\left[f^{1/2}(\vect{x_0}+\rho \vect{v},\vect{\xi})-f^{1/2}(\vect{x_0},\vect{\xi})\right]^2\dif \vect{\xi}\right]^{\sfrac{1}{2}}\;.
\end{align*}
The analysis of the first integral has already be done in the study of  $\mat{I_1}$ taking $\vect{u}=\vect{v}$. We then obtain
\begin{equation}
\label{eq:I2a}
\lim_{\rho\to 0}\rho^{-2h(\vect{x_0})}\int_{\mathbb{R}^2}\left|\expj{\rho \ip{\vect{u}}{\vect{\xi}}}-1\right|^2 f(\vect{x_0}+\rho \vect{u},\vect{\xi})\dif \vect{\xi}=\mathbb{E}[Y_{\vect{x_0}}(\vect{u})^2]\;.
\end{equation}
The bound of the second integral directly comes from Lemma~\ref{lem:2}. Since $\beta \geq \sup_{\vect{x}} h(\vect{x})$, we get that
$$
\lim_{\rho\to 0} \rho^{-2h(\vect{x_0})}\mat{I_2} = 0\;.
$$
\noindent
The same approach also yields for $\mat{I_3}$ and $\mat{I_4}$, leading to 
$$
\lim_{\rho\to 0} \rho^{-2h(\vect{x_0})}\mat{I_3} = \lim_{\rho\to 0} \rho^{-2h(\vect{x_0})}\mat{I_4}= 0\;.
$$
%\begin{align*}
%&\int_{\mathbb{R}^2}(\eu^{\, \jj \rho \ip{\vect{v}}{\vect{\xi}}}-1)(\eu^{-\jj \ip{\vect{x_0}}{\vect{\xi}}}-1)\left[f^{1/2}(\vect{x_0}+\rho \vect{u},\vect{\xi})-f^{1/2}(\vect{x_0},\vect{\xi})\right]f^{1/2}(\vect{x_0}+\rho \vect{v},\vect{\xi})\dif \vect{\xi}=o(\rho^{2h(\vect{x_0})})\;,\\ 
%&\int_{\mathbb{R}^2}\left|\eu^{-\jj \ip{\vect{x_0}}{\vect{\xi}}}-1\right|^2\left[f^{1/2}(\vect{x_0}+\rho \vect{u},\vect{\xi})-f^{1/2}(\vect{x_0},\vect{\xi})\right]\left[f^{1/2}(\vect{x_0}+\rho \vect{v},\vect{\xi})-f^{1/2}(\vect{x_0},\vect{\xi})\right]\dif \vect{\xi}=o(\rho^{2h(\vect{x_0})})\;,
%\end{align*}
which concludes the proof of \eqref{e:tangentfield}.

\

(ii) \underline{Second step}: \\

We now have to prove that the sequence $(Z_n)_{n\in \mathbb{N}^{\ast}}\overset{\mathrm{def}}{=}(Z_{\vect{x_0},\rho_n})_{n\in \mathbb{N}^{\ast}}$ satisfies \eqref{eq:kolmo_criteria}, with $\rho\equiv \rho_n=\frac{1}{n} \in [0,1]$ and with $\vect{u}$ and $\vect{v}$ restricted to $B(\vect{0},1/2)$ from Remark \ref{rem:kolmo}, that in the following $T=1/2$. Then, we have:

%\begin{equation}
%\forall T>0, \forall \vect{u}, \vect{v}\in [-T,T]^{2},\quad \sup\limits_{n>1}\mathbb{E}(|Z_n(\vect{u})-Z_n(\vect{v})|^{\gamma_1})\leq C_{0}(T)\lVert \vect{u}-\vect{v}\rVert^{2+\gamma_2}\;,
%\end{equation}
%for some positive constant $C_{0}(T)$ which may depend on $T$ and $\gamma_1,\gamma_2$ which are universal positive constants.\\

\begin{eqnarray*}
&&Z_{\vect{x_0},\rho}(\vect{u})-Z_{\vect{x_0},\rho}(\vect{v})\\
&=&\frac{1}{\rho^{h(\vect{x_0})}} \int_{ \mathbb{R}^2}^{} \left[\left(\expj{\ip{\vect{x_0}+\rho \vect{u}}{\vect{\xi}}}-1\right)f^{1/2}(\vect{x_0}+\rho \vect{u})-\left(\expj{\ip{\vect{x_0}+\rho \vect{v}}{\vect{\xi}}}-1\right)f^{1/2}(\vect{x_0}+\rho \vect{v})\right]\widehat{\ope{W}}(\diff \vect{\xi})\;.
\end{eqnarray*}
Hence,
\begin{eqnarray*}
&& \mathbb{E}\left[\left(Z_{\vect{x_0},\rho}(\vect{u})-Z_{\vect{x_0},\rho}(\vect{v})\right)^2\right]\\
&=&\frac{1}{\rho^{2h(\vect{x_0})}} \int_{ \mathbb{R}^2}^{} \abs{\left(\expj{\ip{\vect{x_0}+\rho \vect{u}}{\vect{\xi}}}-1\right)f^{1/2}(\vect{x_0}+\rho \vect{u},\vect{\xi})-\left(\expj{\ip{\vect{x_0}+\rho \vect{v}}{\vect{\xi}}}-1\right)f^{1/2}(\vect{x_0}+\rho \vect{v},\vect{\xi})}^2\dif \vect{\xi}\;,\\
&=&\frac{1}{\rho^{2h(\vect{x_0})}} \int_{ \mathbb{R}^2}^{} \biggl|\left(\expj{\ip{\vect{x_0}+\rho \vect{u}}{\vect{\xi}}}-1\right)\left(f^{1/2}(\vect{x_0}+\rho \vect{u},\vect{\xi})-f^{1/2}(\vect{x_0}+\rho \vect{v},\vect{\xi})\right)\\
&&\qquad \qquad \qquad  -\left(\expj{\ip{\vect{x_0}+\rho \vect{v}}{\vect{\xi}}}-1-(\expj{\ip{\vect{x_0}+\rho \vect{u}}{\vect{\xi}}}-1)\right)f^{1/2}(\vect{x_0}+\rho \vect{v},\vect{\xi})\biggl|^2\dif \vect{\xi}\;,\\
&\leq & \frac{2}{\rho^{2h(\vect{x_0})}} \int_{ \mathbb{R}^2}^{} \abs{\expj{\ip{\vect{x_0}+\rho \vect{u}}{\vect{\xi}}}-1}^2\left(f^{1/2}(\vect{x_0}+\rho \vect{u},\vect{\xi})-f^{1/2}(\vect{x_0}+\rho \vect{v},\vect{\xi})\right)^2\diff \vect{\xi}\\
&&+\frac{2}{\rho^{2h(\vect{x_0})}} \int_{ \mathbb{R}^2} \abs{\expj{\ip{\rho (\vect{v}-\vect{u})}{\vect{\xi}}}-1}^2 f(\vect{x_0}+\rho \vect{v})\diff \vect{\xi}\;.
\end{eqnarray*}
the last inequality coming from $\abs{a-b}^2\leq 2\abs{a}^2+2\abs{b}^2$.

We now apply Lemma~ \ref{lem:2} with $\vect{x}=\vect{x_0}+\rho\vect{u}$ and $\vect{w}=\vect{v}-\vect{u}\in B(\vect{0},1)$. It implies the following bound for the first integral
\begin{eqnarray}
&&\frac{2}{\rho^{2h(\vect{x_0})}}\int_{ \mathbb{R}^2}^{} \abs{\expj{\ip{\vect{x_0}+\rho \vect{u}}{\vect{\xi}}}-1}^2\left(f^{1/2}(\vect{x_0}+\rho \vect{u},\vect{\xi})-f^{1/2}(\vect{x_0}+\rho \vect{v},\vect{\xi})\right)^2\diff \vect{\xi}\;,\nonumber\\
&\leq& 2\left(\sup_{\vect{y}\in B(\vect{x_0},2T)}K_{\vect{y}}\right)\left(1+2\|\vect{x_0}\|^2+2\|\vect{u}\|^2\right)\max\left(\|\vect{v}-\vect{u}\|^{2\beta},\|\vect{v}-\vect{u}\|^{2\eta}\right)|\rho|^{2\beta-2h(\vect{x_0})}\;,\nonumber\\
&\leq & K_1 |\rho|^{2\beta-2h(\vect{x_0})} \max\left(\|\vect{v}-\vect{u}\|^{2\beta},\|\vect{v}-\vect{u}\|^{2\eta}\right)\;,\label{e:K1}
\end{eqnarray}
with $\displaystyle K_1=2\left(\sup_{\vect{y}\in B(\vect{x_0},2T)}K_{\vect{y}}\right)\left(1+2\|\vect{x_0}\|^2+4T^2\right)$ and $\beta-h(\vect{x}_0)$ positive.

To bound the second one observe that $C$ is homogeneous w.r.t. the second variable and bounded. Set $\vect{\zeta}=\rho\vect{\xi}\|\vect{u}-\vect{v}\|$ and deduce that
\begin{eqnarray*}
&& \frac{2}{\rho^{2h(\vect{x_0})}} \int_{ \mathbb{R}^2} \abs{\expj{\ip{\rho (\vect{v}-\vect{u})}{\vect{\xi}}}-1}^2 f(\vect{x_0}+\rho \vect{v})\diff \vect{\xi}\\
&\leq& 2\|\vect{u}-\vect{v}\|^{2h(\vect{x_0}+\rho \vect{u})}\|C\|_\infty \rho^{2h(\vect{x_0}+\rho\vect{u})-2h(\vect{x_0})}\\
&&\times \int_{ \mathbb{R}^2} \abs{\expj{\ip{\vect{\Theta}}{\vect{\zeta}}}-1}^2\left(\|\vect{\zeta}\|^{-2a-2} \mathds{1}_{\|\vect{\zeta}\|>1}+\|\vect{\zeta}\|^{-2b-2} \mathds{1}_{\|\vect{\zeta}\|\leq 1}\right)\diff \vect{\zeta}\;,
\end{eqnarray*}
with $\vect{\Theta}=(\vect{u}-\vect{v})/\|\vect{u}-\vect{v}\|$. By the same arguments than in (\ref{eq:lim_rho}) we have that 
\[
\lim_{\rho\to 0} \rho^{2(h(\vect{x_0}+\rho \vect{v})-h(\vect{x_0}))}=1\;,
\]
so it has a finite upper bound $A_1>0$, which is achieved on a compact by continuity
\[
A_1=\max_{\rho, \vect{v}} \left\{(\rho,\vect{v})\in [0,1]\times [-T,T]^2,\; \rho^{2(h(\vect{x_0}+\rho \vect{v})-h(\vect{x_0}))}\right\}\;.
\]
Identically we have $\lVert \vect{u}-\vect{v}\rVert^{2h(\vect{x_0}+\rho \vect{v})}=\lVert \vect{u}-\vect{v}\rVert^{2(h(\vect{x_0}+\rho \vect{v})-h(\vect{x}_0))}\lVert \vect{u}-\vect{v}\rVert^{2h(\vect{x}_0)}$ and the first term tends to 1, then the function $(\rho,\vect{u},\vect{v})\mapsto  \lVert \vect{u}-\vect{v}\rVert^{2(h(\vect{x_0}+\rho \vect{v})-h(\vect{x}_0))}$ also achieves its upper bound $A_2>0$. Thus, 
\[
2\|\vect{u}-\vect{v}\|^{2h(\vect{x_0}+\rho \vect{u})}\|C\|_\infty \rho^{2h(\vect{x_0}+\rho\vect{u})-2h(\vect{x_0})}\leq 2A_1A_2\|C\|_\infty \lVert \vect{u}-\vect{v}\rVert^{2h(\vect{x}_0)}\;.
\]   

%with $M_\beta=\sup_{\rho\in (0,1)} \rho^\beta\log(|\rho|)$ we have
%\[
%\rho^{2h(\vect{x_0}+\rho\vect{u})-2h(\vect{x_0})}\leq \exp(\rho^\beta\log(|\rho|)\|\vect{u}\|^{2\beta})\leq \exp(M_{\beta}(2T)^{\beta})\;.
%\]
Hence, using that $\abs{\expj{\ip{\vect{\Theta}}{\vect{\zeta}}}-1}^2\leq \min(2,\lVert \vect{\Theta}\rVert \lVert \vect{\zeta}\rVert)=\min(2,\lVert \vect{\zeta}\rVert)$,
\begin{equation}\label{e:K2}
\frac{2}{\rho^{2h(\vect{x_0})}} \int_{ \mathbb{R}^2} \abs{\expj{\ip{\rho (\vect{v}-\vect{u})}{\vect{\xi}}}-1}^2 f(\vect{x_0}+\rho \vect{v},\vect{\xi})\diff \vect{\xi}
\leq K_2\|\vect{u}-\vect{v}\|^{2h(\vect{x}_0)}\;,
\end{equation}
with $K_2=2A_1A_2\|C\|_\infty\int_{ \mathbb{R}^2} \min(2,\|\vect{\zeta}\|^2)(\|\vect{\zeta}\|^{-2a-2} \mathds{1}_{\|\vect{\zeta}\|>1}+\|\vect{\zeta}\|^{-2b} \mathds{1}_{\|\vect{\zeta}\|\leq 1})\diff \vect{\zeta}$.\\

Since $K_1,K_2$ are two positive constants depending only on $T$, inequalities~\eqref{e:K1} and~\eqref{e:K2} imply that
\[
\mathbb{E}\left[\left(Z_{\vect{x_0},\rho}(\vect{u})-Z_{\vect{x_0},\rho}(\vect{v})\right)^2\right]\leq K_2 \lVert \vect{u}-\vect{v}\rVert^{2h(\vect{x}_0)}\left[1+\frac{K_1}{K_2} \max\left(\|\vect{v}-\vect{u}\|^{2(\beta-h(\vect{x}_0))},\|\vect{v}-\vect{u}\|^{2(\eta-h(\vect{x}_0))}\right)\right]\;,
\]
and since $\beta-h(\vect{x}_0)>$ and $\eta-h(\vect{x}_0)>0$, the second factor achieved its bounds on the compact set $[-T,T]^2\times [-T,T]^2$, hence
\[
\sup_{\rho\in (0,1)}\lVert \vect{u}-\vect{v}\rVert^{-2h(\vect{x}_0)}\mathbb{E}\left[\left(Z_{\vect{x_0},\rho}(\vect{u})-Z_{\vect{x_0},\rho}(\vect{v})\right)^2\right]<\infty\;.
\]
Which proves the inequality~\eqref{eq:kolmo_criteria} from Remark \ref{rem:gaussian_moment} with $\gamma_0=h(\vect{x}_0)$. \hfill $\square$\\

\noindent The proof of the points (i) and (ii) completes the proof of the Proposition \ref{theo:local:orientation:genCFBA}.
%To conclude, we use the fact that $Z_{\vect{x_0},\rho}(\vect{u})-Z_{\vect{x_0},\rho}(\vect{v})$ is a Gaussian r.v. and then for any $\gamma_1>0$
%\[
%\sup_{\rho\in (0,1)}\|\vect{u}-\vect{v}\|^{-a\gamma_1}\mathbb{E}\left[\left|Z_{\vect{x_0},\rho}(\vect{u})-Z_{\vect{x_0},\rho}(\vect{v})\right|^{\gamma_1}\right]\;,
%\]
%and
%\[
%\left[\sup_{\rho\in (0,1)}\|\vect{u}-\vect{v}\|^{-2a}\mathbb{E}\left|Z_{\vect{x_0},\rho}(\vect{u})-Z_{\vect{x_0},\rho}(\vect{v})\right|^{2}\right]^{\gamma_1/2}\;.
%\]
%are equal up to a multiplicative constant depending only on $\gamma_1$. Inequality~\eqref{eq:kolmo_criteria} then follows if we consider $\gamma_1>2/a$ with $\gamma_2=a\gamma_1-2$. 
%%%%%%%%%%%%%%%%%%%%%%%%%%%%%
\subsection{Proof of Proposition~\ref{pro:ex2}\\}\label{s:proof:pro:ex2}

Let $\vect{x_0}\in \R^2$. Since $X$ is $H$--self-similar, one has 
\[
X(\vect{x})=\int_{\mathbb{R}^2}(\expj{\langle \vect{x},\vect{\xi}\rangle}-1)f^{1/2}(\vect{\xi})~\widehat{\ope{W}}(\diff \vect{\xi})\;,
\]
with $f(\vect{\xi})=C_X(\vect{\xi}) ~ \lVert \vect{\xi}\rVert^{-2H-2}$. Then,
$$
Z_{\ope{\Phi},X}(\vect{x})=X(\ope{\Phi}(\vect{x}))=\int_{\mathbb{R}^2}(\expj{\langle \ope{\Phi}(\vect{x}),\vect{\xi}\rangle}-1)f^{1/2}(\vect{\xi})~\widehat{\ope{W}}(\diff \vect{\xi})\;.
$$

As for the proof of Proposition \ref{theo:local:orientation:genCFBA} in Section \ref{proof:prop9}, we divide the following proof into two steps.

\

\noindent (i) \underline{First step:} 

\

Let $ \vect{u},\vect{v}\in\mathbb{R}^2$, and consider:
\begin{eqnarray*}
&&\frac{1}{\rho^{2H}}\mathbb{E}\left[\left(Z_{\ope{\Phi},X}(\vect{x_0}+\rho \vect{u})-Z_{\ope{\Phi},X}(\vect{x_0})\right)\left(Z_{\ope{\Phi},X}(\vect{x_0}+\rho \vect{v})-Z_{\ope{\Phi},X}(\vect{x_0})\right)\right]\\
&=&\frac{1}{\rho^{2H}} \int_{\mathbb{R}^2}\expj{\ip{\ope{\Phi}(\vect{x_0})}{\vect{\xi}}}\left(\expj{\ip{\ope{\Phi}(\vect{x_0}+\rho \vect{u})-\ope{\Phi}(\vect{x_0})}{\vect{\xi}}}-1\right)\eu^{-\jj \ip{\ope{\Phi}(\vect{x_0})}{\vect{\xi}}}\left(\eu^{-\jj \ip{ \ope{\Phi}(\vect{x_0}+\rho \vect{v})-\Phi(\vect{x_0})}{\vect{\xi}}}-1\right)f(\vect{\xi})\diff \vect{\xi}\;,\\
&=&\frac{1}{\rho^{2H}} \int_{\mathbb{R}^2}
\left(\expj{\ip{\frac{\ope{\Phi}(\vect{x_0}+\rho \vect{u})-\ope{\Phi}(\vect{x_0})}{\rho}}{\rho \vect{\xi}}}-1\right)
\left(\eu^{-\jj \ip{\frac{\ope{\Phi}(\vect{x_0}+\rho \vect{v})-\ope{\Phi}(\vect{x_0})}{\rho}}{\rho \vect{\xi}}}-1\right)
f(\vect{\xi})\diff \vect{\xi}\;,\\
&=&\frac{1}{\rho^{2H}} \int_{\mathbb{R}^2}\left(\expj{\ip{\frac{\ope{\Phi}(\vect{x_0}+\rho \vect{u})-\ope{\Phi}(\vect{x_0})}{\rho}}{\vect{\zeta}}}-1\right)
\left(\eu^{-\jj \ip{\frac{\ope{\Phi}(\vect{x_0}+\rho \vect{v})-\ope{\Phi}(\vect{x_0})}{\rho}}{ \vect{\zeta}}}-1\right)f(\vect{\zeta}/\rho)\diff \vect{\zeta}/\rho^2\;,\\
&=&\frac{1}{\rho^{2H}} \int_{\mathbb{R}^2}\left(\expj{\ip{\frac{\ope{\Phi}(\vect{x_0}+\rho \vect{u})-\ope{\Phi}(\vect{x_0})}{\rho}}{\vect{\zeta}}}-1\right)\left(\eu^{-\jj \ip{\frac{\ope{\Phi}(\vect{x_0}+\rho \vect{v})-\ope{\Phi}(\vect{x_0})}{\rho}}{ \vect{\zeta}}}-1\right)\rho^{2H+2}f(\vect{\zeta})\diff \vect{\zeta}/\rho^2\;,\\
&=&\int_{\mathbb{R}^2}\left(\expj{\ip{\frac{\ope{\Phi}(\vect{x_0}+\rho \vect{u})-\ope{\Phi}(\vect{x_0})}{\rho}}{\vect{\zeta}}}-1\right)\left(\eu^{-\jj \ip{\frac{\ope{\Phi}(\vect{x_0}+\rho \vect{v})-\ope{\Phi}(\vect{x_0})}{\rho}}{ \vect{\zeta}}}-1\right)f(\vect{\zeta})\diff \vect{\zeta}\;.
\end{eqnarray*}To compute the limit of this quantity when $\rho \to 0$, let us denote by $g(\rho,\vect{\zeta})$ the integrand of the last integral. We have
\[
\lim_{\rho\to 0}g(\rho,\vect{\zeta})=\left(\expj{\ip{ \ope{D}\ope{\Phi}(\vect{x_0})\vect{u}}{\vect{\zeta}}}-1\right)\left(\eu^{-\jj \ip{ \ope{D}\ope{\Phi}(\vect{x_0})\vect{v}}{\vect{\zeta}}}-1\right)f(\vect{\zeta})\;.
\]
Now we have to bound the integrand $\lvert g(\rho,\vect{\zeta})\rvert$: using the inequality $\lvert \eu^{\jj x}-1\rvert\leqslant \min (2,\lvert x\rvert)$, one has
\begin{align*}
\lvert g(\rho,\vect{\zeta})\rvert
&\leqslant \min \left(2,\left \lvert \ip{\frac{\ope{\Phi}(\vect{x_0}+\rho \vect{u})-\ope{\Phi}(\vect{x_0})}{\rho}}{\vect{\zeta}}\right \rvert\right)\min \left(2,\left \lvert \ip{\frac{\ope{\Phi}(\vect{x_0}+\rho \vect{v})-\ope{\Phi}(\vect{x_0})}{\rho}}{\vect{\zeta}}\right \rvert\right)f(\vect{\zeta})\;,\\
&\leqslant \min \left(2,\frac{1}{\rho}\norme{\ope{\Phi}(\vect{x_0}+\rho \vect{u})-\ope{\Phi}(\vect{x_0})} \norme{\vect{\zeta}}\right)\min \left(2,\frac{1}{\rho}\norme{\ope{\Phi}(\vect{x_0}+\rho \vect{v})-\ope{\Phi}(\vect{x_0})} \norme{\vect{\zeta}}\right)f(\vect{\zeta})\;,\\
&\leqslant \min \left(2,\frac{1}{\rho}\sup_{[\vect{x_0},\vect{x_0}+\rho \vect{u}]}\norme{\ope{\Phi}'(\vect{x})}\norme{\rho \vect{u}} \norme{\vect{\zeta}}\right)\min \left(2,\frac{1}{\rho}\sup_{[\vect{x_0},\vect{x_0}+\rho \vect{v}]}\norme{\ope{\Phi}'(\vect{x})}\norme{\rho \vect{v}} \norme{\vect{\zeta}}\right)f(\vect{\zeta})\;,\\
&\leqslant \min \left(2,K\norme{\vect{u}} \norme{\vect{\zeta}}\right)\min \left(2,K\norme{\vect{v}} \norme{\vect{\zeta}}\right)f(\vect{\zeta})\;,\\
&\leqslant \min(2,C\norme{\vect{\zeta}})^2f(\vect{\zeta})\equiv G(\vect{\zeta})\;.
\end{align*}
The second inequality is obtained by Cauchy--Schwarz inequality, the third by mean value inequality, the forth under the assumption that $\ope{\Phi}$ which is continuously differentiable so $\norme{\ope{\Phi}'}\leqslant K$, the fifth with $C=K\max(\norme{\vect{u}},\norme{\vect{v}})$. Finally, we show that $G$ is integrable since:
\begin{align*}
\int_{\mathbb{R}^2}G(\vect{\zeta})\dif \vect{\zeta}
&=\int_{\mathbb{R}^2} \min(2,C\norme{\vect{\zeta}})^2f(\vect{\zeta})\dif \vect{\zeta}\;,\\
&=\frac{1}{C^2}\int_{\mathbb{R}^2} \min(2,\norme{\vect{\xi}})^2f\left(\frac{\vect{\xi}}{C}\right)\dif \vect{\xi}\;,\\
&=\frac{C^{2H+2}}{C^2}\int_{\mathbb{R}^2} \min(2,\norme{\vect{\xi}})^2f(\vect{\xi})\dif \vect{\xi}\;,\\
&\leqslant C^{2H}\int_{\mathbb{R}^2} \min(4,\norme{\vect{\xi}}^2) f(\vect{\xi})\dif \vect{\xi}<+\infty\;.
\end{align*}
where we have used  the homogeneity of $f$, and Proposition \ref{prop:borel}. Hence, using the Lebesgue's Dominated Convergence Theorem, we obtain
\[
\lim_{\rho\to 0} \frac{1}{\rho^{2H}}\mathbb{E}[(Z_{\ope{\Phi},X}(\vect{x_0}+\rho \vect{u})-Z_{\ope{\Phi},X}(\vect{x_0}))(Z_{\ope{\Phi},X}(\vect{x_0}+\rho \vect{v})-Z_{\ope{\Phi},X}(\vect{x_0}))]=\mathbb{E}[Y_{\vect{x_0}}(\vect{u})Y_{\vect{x_0}}(\vect{v})]
\]
where we denoted
\[
Y_{\vect{x_0}}(\vect{u})=\int_{\mathbb{R}^2} \left(\expj{\ip{ \ope{D}\ope{\Phi}(\vect{x_0})\vect{u}}{\vect{\xi}}}-1\right)f^{1/2}(\vect{\xi})\ope{\widehat W}(\diff \vect{\xi})\;,
\]
which is by definition the tangent field. 

\

\noindent (ii) \underline{Second step:} 

\

We then prove that the convergence holds in the sense of finite dimensional distributions. To deduce Proposition~\ref{pro:ex2}, we follow the same way as in Step 2 of Theorem~\ref{theo:local:orientation:genCFBA} a Kolmogorov criteria.
\begin{eqnarray*}
&&\mathbb{E}\left[\left(\frac{Z_{\ope{\Phi},X}(\vect{x_0}+\rho \vect{u})-Z_{\ope{\Phi},X}(\vect{x_0})}{\rho^H}-\frac{Z_{\ope{\Phi},X}(\vect{x_0}+\rho \vect{u})-Z_{\ope{\Phi},X}(\vect{x_0})}{\rho^H}\right)^2\right]\\
&=&\frac{1}{\rho^{2H}} \int_{\mathbb{R}^2} \abs{\expj{\ip{\ope{\Phi}(\vect{x_0}+\rho \vect{u})}{\vect{\xi}}}-\expj{\ip{\ope{\Phi}(\vect{x_0}+\rho \vect{v})}{\vect{\xi}}}}^2f(\vect{\xi})\dif \vect{\xi}\;,\\
&=&\frac{1}{\rho^{2H}} \int_{\mathbb{R}^2} \abs{\expj{\ip{\ope{\Phi}(\vect{x_0}+\rho \vect{u})-\ope{\Phi}(\vect{x_0}+\rho \vect{v})}{\vect{\xi}}}-1}^2f(\vect{\xi})\dif \vect{\xi}\;,\\
&=&\int_{\mathbb{R}^2} \abs{\expj{\ip{\frac{\ope{\Phi}(\vect{x_0}+\rho \vect{u})-\ope{\Phi}(\vect{x_0}+\rho \vect{v})}{\rho}}{\vect{\zeta}}}-1}^2f(\vect{\zeta})\dif \vect{\zeta}\;,\\
&=&\int_{\mathbb{R}_+^{\ast}} \frac{ \abs{\expj{s}-1}^2}{s^{2H+1}}\left[ \int_{ \mathbb{S}^1}^{} \bigg| \ip{\frac{\ope{\Phi}(\vect{x_0}+\rho \vect{u})-\ope{\Phi}(\vect{x_0}+\rho \vect{v})}{\rho}}{\vect{\Theta}}\bigg|^{2H} C_X(\vect{\Theta})\dif \vect{\Theta}\right]\dif s\;,\\
&\leq & \bigg\| \frac{\ope{\Phi}(\vect{x_0}+\rho \vect{u})-\ope{\Phi}(\vect{x_0}+\rho \vect{v})}{\rho}\bigg\|^{2H}\left[ \int_{\mathbb{R}_+^{\ast}} \frac{ \abs{\expj{s}-1}^2}{s^{2H+1}}\dif s\right]\left[ \int_{ \mathbb{S}^1}^{}C_X(\vect{\Theta})\dif \vect{\Theta}\right]\;.
%&=& C_H \bigg\| \frac{\ope{\Phi}(\vect{x_0}+\rho \vect{u})-\ope{\Phi}(\vect{x_0}+\rho \vect{v})}{\rho}\bigg\|^{2H}\;,
\end{eqnarray*}
with $s=r\ip{\frac{\ope{\Phi}(\vect{x_0}+\rho \vect{u})-\ope{\Phi}(\vect{x_0}+\rho \vect{v})}{\rho}}{\vect{\Theta}}$. Let denote by $I_H$ the first integral above and $I_C$ the second one. Then, since $\ope{\Phi}$ is $C^1$, one have
\begin{eqnarray*}
&&\mathbb{E}\left[\left(\frac{Z_{\ope{\Phi},X}(\vect{x_0}+\rho \vect{u})-Z_{\ope{\Phi},X}(\vect{x_0})}{\rho^H}-\frac{Z_{\ope{\Phi},X}(\vect{x_0}+\rho \vect{u})-Z_{\ope{\Phi},X}(\vect{x_0})}{\rho^H}\right)^2\right]\\
&\leq &\frac{I_H I_C}{\rho^{2H}} \left(\displaystyle \sup_{[\vect{x_0}+\rho \vect{u},\vect{x_0}+\rho \vect{v}]} \lVert \ope{\Phi}'(\vect{x})\rVert \lVert \rho (\vect{u}-\vect{v})\rVert\right)^{2H}\;,\\
&\leq &I_H I_C \lVert \ope{\Phi}'\rVert^{2H} \lVert \vect{u}-\vect{v}\rVert^{2H}\;,\\
&\leq  &C_0 \lVert \vect{u}-\vect{v}\rVert^{2H}\;.
\end{eqnarray*}
with $C_0=I_H I_C \lVert \ope{\Phi}'\rVert^{2H}$. We conclude like at the end of Step 2 of Theorem~\ref{theo:local:orientation:genCFBA}.

%\subsection{Proof of Proposition~\ref{pro:ex3}}\label{s:proof:pro:ex3}

\section*{Acknowledgements}
The authors acknowledge the support of the French Agence Nationale de la Recherche (ANR) under reference ANR-13-BS03-0002-01 (ASTRES).

\section*{References}

\bibliographystyle{elsarticle-num}
\bibliography{Orientation}
\end{document}